\documentclass[reqno]{amsart}

\usepackage{enumerate}
\usepackage{amsfonts,amssymb,amsmath,amsthm}
\usepackage{epsfig}
\usepackage{graphics}
\usepackage[normalem]{ulem}
\usepackage{color}
\usepackage{version}

\input xy 
\xyoption{all}

\def\C{\mathbb C}
\def\P{\mathbb P}
\newcommand{\Ctwo}{\C^2}

\def\B{\mathbb B}
\def\bcases{\begin{cases}}

\def\ecases{\end{cases}}

\newcommand{\D}{\mathbb D}

\newcommand{\N}{\mathbb N}

\newcommand{\R}{\mathbb R}

\newcommand{\Z}{\mathbb Z}

\newtheorem{thm}{Theorem}[section]
\newtheorem{main theorem}[thm]{Main Theorem}
\newtheorem{corollary}[thm]{Corollary}

\newtheorem{lemma}[thm]{Lemma}
\newtheorem{prop}[thm]{Proposition}

\newtheorem{problem 1}{Problem 1}
\newtheorem{problem 2}{Problem 2}
\newtheorem{problem 3}{Problem 3}
\theoremstyle{definition}

\newtheorem{defn}[thm]{Definition}
\newtheorem{remark}[thm]{Remark}

\newtheorem{example}[thm]{Example}

\newcommand{\bea}{\begin{eqnarray*}}
\newcommand{\eea}{\end{eqnarray*}}

\newcommand{\be}{\begin{equation}}
\newcommand{\ee}{\end{equation}}

\newcommand{\ra}{\rightarrow}

\renewcommand{\Re}{\mathrm{Re}\,}

\begin{document}

\title
{Dynamics of transcendental H\'enon maps}

\author[L. Arosio]{Leandro Arosio$^{\dag}$}
\author[A.M. Benini]{Anna Miriam Benini$^{\dag}$}
\author[J.E.  Forn{\ae}ss ]{John Erik Forn{\ae}ss$^{*}$}
\author[H. Peters]{Han Peters}

\date{\today}

\thanks{$^{\dag}$  Supported by the SIR grant ``NEWHOLITE - New methods in holomorphic iteration'' no. RBSI14CFME}
\thanks{$^{*}$  Supported by the NFR grant no. 10445200}
\thanks{Part of this work was done during the international research program "Several Complex Variables and Complex Dynamics"
at the Center for Advanced Study at the Academy of Science and Letters in Oslo during the academic year 2016/2017. }
\address{ L. Arosio: Dipartimento Di Matematica\\
Universit\`{a} di Roma \textquotedblleft Tor Vergata\textquotedblright\  \\
 Italy} \email{arosio@mat.uniroma2.it}
\address{ A.M. Benini: Dipartimento Di Matematica\\
Universit\`{a} di Roma \textquotedblleft Tor Vergata\textquotedblright\  \\
 Italy} \email{ambenini@gmail.com}
\address{ H. Peters: Korteweg de Vries Institute for Mathematics\\
University of Amsterdam\\
the Netherlands} \email{hanpeters77@gmail.com}
\address{ J.E. Fornaess: Department of Mathematical Sciences\\
NTNU Trondheim, Norway} \email{john.fornass@ntnu.no}

\begin{abstract}
The dynamics of transcendental functions in the complex plane has received a significant amount of attention. In particular much is known about the description of Fatou components. Besides the types of periodic Fatou components that can occur for polynomials, there also exist so-called Baker domains, periodic components where all orbits converge to infinity, as well as wandering domains.

In trying to find analogues of these one dimensional results, it is not clear which higher dimensional transcendental maps to consider. In this paper we find inspiration from the extensive work on the dynamics of complex H\'enon maps. We introduce the family of transcendental H\'enon maps, and study their dynamics, emphasizing the description of Fatou components. We prove that the classification of the recurrent invariant Fatou components is similar to that of polynomial H\'enon maps, and we give examples of Baker domains and wandering domains.
\end{abstract}
\maketitle
\tableofcontents

\section{Introduction}
Our goal is to combine ideas from two separate areas of holomorphic dynamics: the study of transcendental dynamics on the complex plane, and the study of polynomial H\'enon maps in $\mathbb C^2$. Recall that a polynomial H\'enon map is a map of the form
$$
F: (z,w) \mapsto (f(z) - \delta w, z),
$$
where $f$ is a polynomial of degree at least $2$, and $\delta$ is a non-zero constant. Here we consider maps of the same form, but where $f$ is a transcendental entire function. We call such $F$ a   \emph{ transcendental H\'enon map}, and it is easy to see that $F$ is a holomorphic automorphism of $\C^2$ with constant Jacobian determinant $\delta.$

The main reason for considering transcendental H\'enon maps and not  arbitrary entire maps in $\mathbb C^2$ is that the space of entire maps is too large. Even the class of polynomials maps in two complex variables is often considered too diverse to study the dynamics of these maps all at the same time.
On the other side, the family of polynomial  \emph{ automorphisms} of $\mathbb C^2$ has received a large amount of attention. It portrays a wide variety of dynamical behavior, yet it turns out that this class of maps is homogeneous enough to describe its dynamical behavior in detail. A result of Friedland and Milnor \cite{FM} implies that any polynomial automorphism with non-trivial dynamical behavior is conjugate to a finite composition of polynomial H\'enon maps. It turns out that finite compositions of polynomial H\'enon maps behave in many regards similarly to single H\'enon maps, and the family of H\'enon maps is sufficiently rigid to allow a thorough study of its dynamical behavior.

%

Very little is known about the dynamics of holomorphic automorphisms of $\mathbb C^2$, although there have been results showing holomorphic automorphisms of $\mathbb C^2$ with interesting dynamical behavior, such as the construction of oscillating wandering domains by Sibony and the third named author \cite{FS}, and a result of Vivas, Wold and the last author \cite{PVW} showing that a generic volume preserving automorphisms of $\mathbb C^2$ has a hyperbolic fixed point with a  stable manifold which is dense in $\C^2$.
Transcendental H\'enon maps seems to be a natural class of  holomorphic automorphisms of $\mathbb C^2$ with non-trivial dynamics,  restrictive enough to allow for a clear description of its dynamics, but large enough to display interesting dynamical behaviour which does not appear in the polynomial H\'enon case.

We  classify  in Section \ref{sectionrecurrent} the invariant  \emph{recurrent} components of the Fatou set of a transcendental H\'enon map, that is, components which admits an orbit accumulating to an interior point.  Invariant recurrent components have been described for polynomial H\'enon maps in \cite{BS1991};  our classification holds not only for transcendental H\'enon maps but also for the larger class of holomorphic automorphisms with constant Jacobian. Moreover, using the fact that $f$ is a transcendental holomorphic function, we obtain in Section \ref{sectioninvariant} results about periodic points and invariant algebraic curves. We show that the set ${\rm Fix}(F^2)$  is discrete, and (if $\delta\neq -1$) that $F$ admits  infinitely many   saddle points of period 1 or 2, which implies that the Julia set is not empty. We also show that there is no irreducible  invariant algebraic curve (the same was proved by Bedford-Smillie for polynomial H\'enon maps in \cite{BS91a}).
The dynamical behavior can be restricted even further by considering transcendental H\'enon maps whose map $f$ has a given order of growth. For example, if the order of growth is smaller than $\frac{1}{2}$, then  ${\rm Fix}(F^k)$ is discrete for all $k\geq 1$.

We then give examples of \emph{Baker domains}, \emph{escaping wandering domains},  and \emph{oscillating wandering domains}. Such Fatou components appear in transcendental dynamics in $\C$, and for trivial reasons they cannot occur for polynomials. The existence of the filtration gives a similar obstruction for polynomial H\'enon maps, but this filtration is lost when considering transcendental H\'enon maps.

For a transcendental function a Baker domain is a periodic  Fatou component on which the orbits converge locally uniformly  to  the point $\infty$, which is an essential singularity \cite{BerSur}.
We give an example in Section \ref{sectionbaker}  of a transcendental H\'enon map with a two-dimensional analogue:  a Fatou component on which the orbits converge to a  point at the line at infinity $\ell^\infty$, which is (in an appropriate sense) an essential singularity. In one complex variable for any   Baker domain there exists an absorbing domain, equivalent to a half plane $\mathbb H$, on which the dynamics is conjugate to an affine function, and the conjugacy extends as a semi-conjugacy to the entire Baker domain. In our example the domain is equivalent to $\mathbb H \times \mathbb C$, and the dynamics is conjugate to an affine map.

The final part of the paper is devoted to  wandering domains. Recall that wandering domains are known not to exist for one-dimensional polynomials and rational maps \cite{Sullivan}, but they do arise for transcendental maps (see for example \cite{BerSur}). In higher dimensions it is known that wandering domains can occur for holomorphic automorphisms of $\C^2$ \cite{FS} and for  polynomial maps \cite{ABDPR2016}, but whether polynomial H\'enon maps can have wandering domains remains an open question. We will consider two types of wandering Fatou components, each with known analogues in the one-dimensional setting.
 We construct in Section \ref{sectionescaping}  a  wandering domain,  biholomorphic to $\mathbb C^2$, which is escaping: all orbits converge to the point $[1:1:0]$ at infinity. The construction is again very similar to that in one dimension. However, the proof that the domain and its forward images  are actually different Fatou components is not the proof usually given in one dimension. Instead of finding explicit sets separating one component from another, we give an argument that uses exponential expansion near the boundary of each of the  domains.

Finally, we construct   in Section \ref{sectionoscillating} a transcendental H\'enon map $F$ with a  wandering domain $\Omega$, biholomorphic to $\C^2$, which is oscillating, that is it contains points whose orbits  have both  bounded subsequences and  subsequences which converge to infinity.
Up to a linear change of variable, the map $F$ is  the limit as $k\to\infty$ of automorphisms of $\C^2$ of the form $F_k(z,w):=(f_k(z)+\frac{1}{2}w,\frac{1}{2}z)$, all having a hyperbolic fixed point at the origin.
The family $(F_k)$ is constructed inductively using Runge approximation in one variable to obtain an entire function $f_{k+1}$ which  is   sufficiently close to $f_k$ on larger and larger disks, in such a way  that the orbit  of an open set $U_0\subset \C^2$ approaches the origin coming in along the stable manifold of $F_k$  and then goes outwards along the unstable manifold of $F_k$, over and over for all $k\in \N$.

Regarding the complex structure of those Fatou components, in both the Baker domain and the oscillating wandering domain case one encounters the same difficulty. Namely, in both cases one finds a suitable invariant domain $A$ of the Fatou component on which it is possible to construct, using the dynamics of $F$, a biholomorphism to a model space ($\mathbb{H}\times \C$ and $\C^2$ respectively, where $\mathbb{H}$ denotes the right half-plane). One then needs to prove that the domain $A$ is in fact  the whole Fatou component, and this is done by using the following \emph{plurisubharmonic method}: If $A$ is strictly smaller than $\Omega$ then
we can construct a plurisubharmonic function $u\colon \Omega\to \R\cup\{-\infty\}$ for which the submean value property is violated at points in $\partial A\cap \Omega$. We note that a somewhat similar argument was given by the third author in \cite{shortC^k}, and we believe that this method can be applied in a variety of similar circumstances.

It is important to point out that for an entire  map $F\colon \C^2\to\C^2$  there are two natural definitions of the Fatou set, which  correspond to
compactifying $\C^2$ either with the one-point compactification $\widehat {\C^2}$, or with $\P^2$.
In one dimension the two Fatou sets coincide, and the same is true for polynomial H\'enon maps,  since by the existence of the filtration all forward orbits that converge to infinity converge to the same point on the line at infinity $\ell^\infty = \mathbb P^2 \setminus \mathbb C^2$.
For a general entire self-map  of $\C^2$ these two definitions can give  two different Fatou sets (see Example \ref{exampledifferent}).
Notice that,  if we compactify with  $\widehat {\C^2}$, any open subset of $\C^2$ on which the sequence of   iterates $F^n$ diverges uniformly on compact subsets would be  in the Fatou set regardless of \emph{how} the orbits go to infinity. This seems to be too weak a definition in two complex variables. We thus  define the Fatou set compactifying $\C^2$ with $\P^2$ (which  has the additional advantage of being a complex manifold).
Section \ref{sectionfatou} is devoted to this argument.







\section{The definition of the Fatou set}\label{sectionfatou}

Let $n\in \N$ and let $X$ be a complex manifold.
There are (at least) two natural definitions of what it means for a family $\mathcal{F}\subset {\rm Hol}(X,\C^n)$ to be normal.
We denote by $\widehat {\C^n}$ the one-point compactification of $\C^n$, and with the symbol $\infty$ we denote both the point at infinity and the constant map $z\mapsto\infty$.
\begin{defn}
A family $\mathcal{F}\subset {\rm Hol}(X,\C^n)$ is  \emph{ $\P^n$-normal}  if for every sequence $(f_n)\in\mathcal{F}$ there exists a subsequence $(f_{n_k})$ converging uniformly on compact subsets to $f\in {\rm Hol}(X,\P^n)$.
In other words, $\mathcal{F}$ is pre-compact in ${\rm Hol}(X,\P^n)$.

A family $\mathcal{F}\subset {\rm Hol}(X,Y)$ is  \emph{ $\widehat{\C^n}$-normal} if for every sequence $(f_n)\in\mathcal{F}$ which is not divergent on compact subsets there exists a subsequence
$(f_{n_k})$ converging uniformly on compact subsets to $f\in {\rm Hol}(X,\C^n)$.
This is equivalent to $\mathcal{F}$ being pre-compact in ${\rm Hol}(X,\C^n)\cup \infty\subset C^0(X,\widehat{\C^n})$.
\end{defn}

\begin{remark}
When $n=1$ the two definitions are equivalent.
\end{remark}

A family $\mathcal{F}\subset {\rm Hol}(X,\C^n)$ is $\P^n$-normal  if and only if it is equicontinuous with respect to the Fubini-Study distance on $\P^n$. This follows from the Ascoli-Arzel\`a theorem and from the fact that  ${\rm Hol}(X,\P^n)$ is closed in $C^0(X,\P^n)$.
One may think  that, similarly, a  family $\mathcal{F}\subset {\rm Hol}(X,\C^n)$ is $\widehat{\C^n}$-normal if and only if it is equicontinuous with respect to the spherical distance $d_{\widehat {\C^n}}$ on $\widehat {\C^n}$, but  this is not the case, as the following example shows.

\begin{figure}[t]
\centering
\includegraphics[width=2.5in]{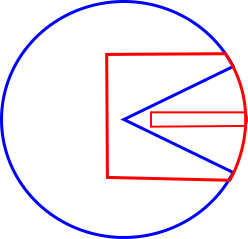}
\caption{The mouth of Pac-man $P_{n+1}$ is contained in bait $R_n$.}
\label{figure:pacman}
\end{figure}

\begin{example}
For $n\geq 2$, the family   ${\rm Hol}(\D,\C^n)\cup \infty$ is not closed in $C^0(\D,\widehat{\C^n})$.
As a consequence, for a family  $\mathcal{F}\subset {\rm Hol}(\D,\C^n)$, being pre-compact in $C^0(\D,\widehat{\C^n})$ is not equivalent to being $\widehat{\C^n}$-normal.
\end{example}
\begin{proof}
Let $n=2$.
Let $ s_n\geq 0$ be an increasing sequence of real numbers converging to $\frac{1}{2}$. Let $ \alpha_n\leq \frac{\pi}{2}$ be a decreasing sequence converging to 0.
We define the  \emph{ Pac-man} $$P_n:=\overline \D\setminus \{s_n+\rho e^{i\theta}\colon  \rho> 0, | \theta|< \alpha_n\}.$$
Let $r_n\geq 0$ be an increasing sequence converging to $\frac{1}{2}$.  Let $\beta_n$  be a sequence decreasing to 0.
We define the  \emph{ bait} $$R_n:=\overline \D\cap \{z\in \C\colon |{\rm Im}\, z|\leq  \beta_n, {\rm Re}\, z\geq r_n\}.$$

Clearly $$\bigcap_{n\in \N}R_n=\{z\in \C\colon |{\rm Im}\, z|=0, \frac{1}{2}\leq {\rm Re}\, z\leq 1\},$$ which we call the  \emph{ slit} $S$.

We  can choose the sequences $(s_n), (\alpha_n), (r_n), (\beta_n)$ in such a way that
\begin{align*}
&P_n\cap R_n=\varnothing\\
&\D\setminus P_{n+1} \subset R_n.
\end{align*}
Notice that this implies that $s_n\leq r_n\leq s_{n+1}$ for all $n\in \N$. See Figure \ref{figure:pacman} for an illustration of a single Pac-man $P_{n+1}$ and two baits $R_{n+1}$ and $R_n$.

Let $b_i>0$ be a sequence such that for all $n\in \N$, $x\in \C^2$ the following implication holds
\begin{equation}\label{neighbinfty}
\|x\|+\frac{1}{2^n}\geq \sum_{i=0}^{n-1} b_i \Longrightarrow d_{\widehat \C^2}(x,\infty)\leq \frac{1}{2^{n+1}}.
\end{equation}
By using Runge approximation we can  define a sequence of holomorphic  functions $f_n: D(0,1+\epsilon )\ra\C$ such that $|f_n|\leq\frac{1}{2^{n+1}}$ on $P_{2n}$, and ${\sf Re}\,f_n\geq b_n$ on $R_{2n}$,
and a sequence of holomorphic functions $g_n:D(0,1+\epsilon )\ra\C$ such that $|g_n|\leq\frac{1}{2^{n+1}}$ on $P_{2n+1}$, and ${\sf Re}\,g_n\geq b_n$ on $R_{2n+1}$.
For all $n\geq 0$, let $H_n:D(0,1+\epsilon )\ra\C^2$ be defined as $H_n(z)=\sum_{j=0}^n (f_n,g_n)$.
Denote $$a_n:=\max_{z\in \overline \D} d_{\widehat \C^2}(H_{n}(z), H_{n-1}(z)).$$ We claim that $a_n\leq \frac{1}{2^n}$, and thus the sequence $(H_n)$ converges uniformly on the disk $\overline \D$ to a continuous mapping $H\colon \overline\D\to \widehat \C^2$ such that
$H(\overline \D\setminus S)\subset \C^2$ and $H(S)=\{\infty\}$.
Fix $n\in \N$. If $z\in P_{2n}$, then $|f_n(z)|\leq \frac{1}{2^{n+1}}$ and $|g_n(z)|\leq \frac{1}{2^{n+1}}$. Thus $\|(f_n(z),g_n(z))\|\leq \frac{1}{2^n}$, which implies that
$d_{\widehat \C^2}(H_{n}(z), H_{n-1}(z))\leq \frac{1}{2^n}.$
If $z\in \D\setminus (P_{2n}\cup R_{2n})$, then we have that ${\sf Re}\sum_{i=0}^{n-1} g_i (z)\geq \sum_{i=0}^{n-1} b_i$ and
$|g_n(z)|\leq \frac{1}{2^n}$, and thus by (\ref{neighbinfty})
both $H_{n}(z)$ and $H_{n-1}(z)$ belong to the ball of radius $\frac{1}{2^{n+1}}$ centered at $\infty$.
If $z\in R_{2n}$, then we have that ${\sf Re}\sum_{i=0}^{n-1} f_i (z)\geq \sum_{i=0}^{n-1} b_i$  and ${\sf Re}\,f_n(z)\geq b_n$,   and thus by (\ref{neighbinfty}) both $H_{n}(z)$ and $H_{n-1}(z)$ belong to the ball of radius $\frac{1}{2^{n+1}}$ centered at $\infty$.
\end{proof}

\begin{lemma}\label{projectiveimpliesintrinsic}
If a   family $\mathcal{F}\subset {\rm Hol}(X,\C^n)$ is $\P^n$-normal, then it is $\widehat{\C^n}$-normal.
\end{lemma}
\begin{proof}
Let $(f_n)$ be a sequence in $\mathcal{F}$. Since $\mathcal{F}$ is $\P^n$-normal there exists a subsequence $(f_{n_k})$ converging uniformly on compact subsets to a map $f\in {\rm Hol}(X,\P^n).$
If there is a  point $x\in X$ such that $f(x)\in \ell^\infty$, then $f(X)\subset \ell^\infty$.
Indeed, it suffices to show that $f^{-1}(\ell^\infty)$ is open, and this follows taking an affine chart around  $f(y)\in \ell^\infty$ in such a way that  $\ell^\infty=\{z_1=0\}$ and applying Hurwitz theorem to the sequence $\pi_1\circ f_n$.

Thus, if the sequence $(f_n)$ is not diverging on compact subsets, the subsequence $(f_{n_k})$ converges uniformly on compact subsets to a map $f\in {\rm Hol}(X,\C^n).$

\end{proof}

As a consequence of the previous discussion, for an entire map $F\colon \C^n\to \C^n$ we have two possible definitions of the Fatou set.
\begin{defn}
A point $z\in\C^n$ belongs to the  \emph{ $\widehat{\C^n}$-Fatou set}  if the family of iterates $(F^n)$ is  $\widehat{\C^n}$-normal near $z$.
A point $z\in\C^n$ belongs to the  \emph{ $\P^n$-Fatou set}  if the family of iterates $(F^n)$ is $\P^n$-normal near $z$.
\end{defn}

By Lemma \ref{projectiveimpliesintrinsic} the $\P^n$-Fatou set is contained in the $\widehat{\C^n}$-Fatou set, but if $n>1$  the inclusion  may be strict  as the following example shows.

\begin{example}\label{exampledifferent}
Given an increasing sequence $N_j \in \mathbb N$, and consider the sequence of polynomials
$$
f_j(z) = (z-5(j-1))^{N_j},
$$
defined respectively on the disks $D_j = D(5(j-1),2)$, where $j\geq 1$. Given a sequence $\epsilon_j \searrow 0$, by Runge approximation (see e.g. Lemma \ref{rungeappr}) we can find an entire function $f$ that is $\epsilon_j$-close to $f_j$ on $\overline D_j$ for all $j$.

Define the map $F \in \mathrm{Aut}(\C^2)$ by
$$
F(z,w) = (z+5, w + f(z)).
$$
It follows immediately from the first coordinate that the forward orbit of any point $(z_0,w_0)$ converges to infinity, ie. $\|F^n(z_0, w_0)\| \rightarrow \infty$, hence the $\widehat{\C^2}$-Fatou set equals all of $\mathbb C^2$. Moreover, if $|z_0| < 1$ then $F^n(z_0, w_0) \rightarrow [1:0:0]$, uniformly on compact subsets. Thus, the domain $\D \times \mathbb C$  is contained in a $\P^2$-Fatou component.

On the other hand, if the sequence $N_j$ increases sufficiently fast, then for $1 < |z_0| \le 2$ we have that $F^n(z_0, w_0) \rightarrow [0:1:0] \in \ell^\infty$, again uniformly on compact subsets. It follows that $\D \times \mathbb C$ is a $\P^2$-Fatou component. Therefore in this example the single $\widehat{\C^2}$-Fatou component contains infinitely many distinct $\P^2$-Fatou components.
\end{example}

\begin{remark}
For a polynomial H\'enon map, it follows from the existence of the invariant filtration that any forward orbit that converges to infinity must converge to  the point $[1:0:0] \in \ell^\infty$. Thus, the two definitions of Fatou set coincide.
\end{remark}


In what follows, we will only consider $\P^2$-normality.
We will call the $\P^2$-Fatou set simply the  \emph{ Fatou set}. The  \emph{ Julia set} is the complement of the Fatou set.

\section{Invariant subsets}\label{sectioninvariant}

\subsection{Periodic points}
%

If $\ell$ is  a transcendental function or a  polynomial H\'enon map, then, for each $k \geq 1$, the set ${\rm Fix}(\ell^k)$ is discrete. Clearly this statement is not satisfied for holomorphic automorphisms of $\mathbb C^2$. For example, one can consider any holomorphic conjugate of a rational rotation.

Consider a periodic orbit
$$
(z_0, w_0) \mapsto (z_1, w_1) \mapsto \cdots \mapsto (z_k, w_k) = (z_0, w_0)
$$
Since $w_{j+1} = z_j$ for each $j$, the first coordinate function of the H\'enon map gives the following relations
\begin{equation}\label{system}
\begin{cases}
f(z_0)  = z_1 + \delta z_{k-1}\\
f(z_1)  = z_2 + \delta z_0\\
 \quad \vdots\\
f(z_{k-1})  = z_0 + \delta z_{k-2}.
\end{cases}
\end{equation}

\begin{lemma}
If $F$ is a transcendental H\'enon map, then  ${\rm Fix}(F)$ and  ${\rm Fix}(F^2)$ are discrete.
\end{lemma}
\begin{proof}
The fixed points $(z,w)$ of $F$ satisfy $z = w$ and thus $z = f(z)-\delta z$. Since $f$ is not linear the set of solutions is discrete.

When $k=2$ the system \eqref{system} gives

\begin{equation}\label{Periodtwo}
\begin{cases}
f(z_0)  = (1+\delta) z_1 \\
f(z_1)  = (1+\delta) z_0.
\end{cases}
\end{equation}

When $\delta = -1$ it is immediate that the set of solutions is discrete. When $\delta \neq -1$ the solutions satisfy

\begin{equation}\label{Periodtwo2}
\begin{cases}
 \frac{f(\frac{f(z_0)}{1+\delta})}{1+\delta}=z_0\\
z_1=\frac{f(z_0)}{1+\delta}.
\end{cases}
\end{equation}

and  again one observes a discrete set of solutions.
\end{proof}

Without making further assumptions it is not clear to the authors that ${\rm Fix}(F^k)$ is discrete when $k\geq 3$. However, we can show discreteness when we assume that the function $f$ has small order of growth.


\begin{prop}\label{prop:discrete}
Let $F$ be a transcendental H\'enon maps such that  $f$ has order of growth strictly less than $\frac{1}{2}$.
Then  ${\rm Fix}(F^k)$ is discrete for all $k\geq 1$.
\end{prop}
\begin{proof}
Consider the entire function
$$
g(z):= \frac{f(z) - f(0)}{z}.
$$
Write
$$
m(r) := \inf_{|z| = r} |g(z)|.
$$
Since $f$ is assumed to have order of growth strictly less than $\frac{1}{2}$, so does $g$, and Wiman's Theorem \cite{wiman} implies that there exist radii $r_n \rightarrow \infty$ for which $m(r_n)\rightarrow \infty$.

Suppose for the purpose of a contradiction that the solution set in $\mathbb C^k$ of the system \eqref{system} is not discrete. Then there exists an unbounded connected component $V$. Let $n \in \mathbb N$ be such that $V$ intersects the polydisk $D(0,r_n)^k$. Then $V$ also intersects the boundary $\partial D(0,r_n)^k$, say in a point $(z_0, \ldots , z_{k-1})$. By the symmetry of the equations in \eqref{system} we may then well assume that $|z_0| = r_n$, and of course that $|z_j| \le r_n$ for $j= 1, \ldots , k-1$.

By Wiman's Theorem we may assume that $|g(z_0)|$ is arbitrarily large, and in particular that $|f(z_0)| > (1+|\delta|) r_n$. But this contradicts the first equation in \eqref{system}, completing the proof.
\end{proof}

We now turn to the question whether transcendental H\'enon maps always have periodic points.
Recall that if $f$ is an  entire transcendental function, then  ${\rm Fix}(F^2)$ has infinite cardinality by \cite{Rosenbloom}.

\begin{prop}\label{Existence2}
If $F(z,w)=(f(z)-\delta w,z)$  is a transcendental H\'enon map, then ${\rm Fix}(F^2)\neq \varnothing$ unless if $\delta=-1$ and $f(z)=e^{h(z)}$ for some holomorphic function $h(z)$.
 If $\delta\neq -1,$ the set  ${\rm Fix}(F^2)\neq\varnothing$ has infinite cardinality.
   \end{prop}
\begin{proof}
Let $Z:=\{f=0\}\subset \C$.  If $\delta=-1$ the solutions of the system (\ref{Periodtwo}) is given by the points in $Z\times Z\subset \C^2$.  In particular there are no solution if and only if $f(z)=e^{h(z)}$ for some holomorphic function $h(z)$.

 If $\delta\neq -1$, let  $g(z):=\frac{f(z)}{1+\delta}$.   The solutions of the  system  (\ref{Periodtwo2})
are the points  $$\{(z_0,g(z_0))\in \C^2\colon z_0\in  {\rm Fix}(g^2)\},$$
which is non-empty and has infinite cardinality. \end{proof}

\begin{remark}
Notice that the set  $Z$  has finite cardinality if and only if $f$ is of the form  $ f(z) =  p(z) e^{h(z)}$, where $p$ is a nonzero polinomial $p$  and $h$ is entire function.  Thus in all other cases, even if $\delta=-1$, the set ${\rm Fix}(F^2)$ has infinite cardinality.
\end{remark}

If $f:\C\ra\C$ is an entire transcendental function, we have additional information on the multiplier of repelling periodic points of period $n\geq 2$. Indeed we have the following theorem \cite[Theorem 1.2]{Ber2}:
\begin{thm}\label{Multipliers}
Let $f$ be a transcendental entire function and let $n\in\N$, $n\geq2$.
Then $f$ has a sequence $z_k$ of periodic points of period $n$ such that
\begin{equation}
|(f^n)'(z_k)|\ra\infty   \text{\ as \ $  k\ra\infty$.}
\end{equation}
\end{thm}

\begin{corollary}[Non-empty Julia set] If $\delta\neq -1$, then  $F$ admits infinitely many   saddle points of period $1$ or $2$, and thus
its Julia set  is non-empty.
\end{corollary}

\begin{proof}
We have seen that for all $z_0\in {\rm Fix}(g^2)$, the point  $(z_0,g(z_0))\in{\rm Fix}(F^2)$.
A computation using the explicit form for $F$ gives
\begin{displaymath}
d_{(z_0,g(z_0))}F^2= \left( \begin{array}{cc}
f'(g(z_0))\cdot f'(z_0) -\delta & -\delta f'(g(z_0))  \\
f'(z_0) & -\delta    \end{array} \right)
\end{displaymath}

Since  $${\rm det}\, d_{(z_0,g(z_0))}F^2= \delta^2, \quad {\rm tr}\, d_{(z_0,g(z_0))}F^2=f'(g(z_0))\cdot f'(z_0) -2\delta,$$ the point  $(z_0,g(z_0))$  is a periodic saddle point  for $|f'(g(z_0))\cdot f'(z_0)|$ sufficiently large (Observe that $g'(g(z_0))\cdot g'(z_0)$, and hence $f'(g(z_0))\cdot f'(z_0)$,  can be taken  arbitrarily large by Theorem~\ref{Multipliers}).

\end{proof}

\subsection{Invariant algebraic curves}
It follows from a result of Bedford-Smillie \cite{BS91a} that a polynomial H\'enon map does not have any invariant algebraic curve. Indeed, given any algebraic curve, the normalized currents of integration on the push-forwards of this curve converge to the $(1,1)$ current $\mu^-$, whose support does not lie on an algebraic curve.

This type of argument is not available for transcendental dynamics. Here we present a different argument.

\begin{thm}\label{thm:invariantcurve}
Let $F$ be a holomorphic automorphism of the form
$$
F: (z,w) \mapsto (f(z) - \delta w, z),
$$
where $f$ is an entire function, and assume that $F$ leaves invariant an irreducible algebraic curve $\{H(z,w) = 0\}$. Then $f$ is affine.
\end{thm}

As we remarked earlier, the statement is known when $f$ is a polynomial of degree at least $2$, so we will assume that $f$ is a transcendental entire function and obtain a contradiction. Let us first rule out the simple case where $\{H = 0\}$ is given by a graph $\{z = g(w)\}$. In that case the invariance under $F$ gives
$$
f(g(w)) - \delta w = g \circ g(w).
$$
Writing $f(z) = g(z)+h(z)$ gives
$$
h\circ g(w) = \delta w,
$$
which implies that $g$ and $h$ are invertible and thus affine. But then $f$ is also affine and we are done.

For a graph of the form $\{w = g(z)\}$ we obtain the functional equation
$$
z = g(f(z) - \delta g(z)),
$$
which again implies that the function $g$ is affine, and then so is $f$.

For the general  case $\{H = 0\}$, where we may now assume that we are not dealing with a graph, we will use the following two elementary estimates.

\begin{lemma}\label{lemma:fastgrowth}
There exist $(z_j, w_j) \in \{H = 0\}$, with $|z_j| \rightarrow \infty$, for which
$$
|f(z_j)| > |z_j|^j.
$$
\end{lemma}
\begin{proof}
As we have already shown that $\{H = 0\}$ is not a graph, it follows that $\{H = 0\}$ intersects all but finitely many lines $\{z = c\}$. The result follows from the assumption that $f$ is transcendental.
\end{proof}

We use two forms for the polynomial $H$:
\begin{enumerate}
\item[(1)] $H(z,w) = p(w)z^{N_1} + \sum_{k=0}^{N_1-1}\sum_{\ell=0}^{N_2}\alpha_{k,\ell}z^kw^\ell.$\\
\item[(2)] $H(z,w) = q_0(z)+\sum_{\ell=1}^n q_\ell(z)w^\ell.$
\end{enumerate}
Note that $q_0$ cannot vanish identically, because otherwise $w$ is a factor of $H$ and the zero set is not irreducible.

%

\begin{lemma} \label{secondlemma}
There exist $d$ large enough so that if $H(z,w)=0$ for $|z|$ sufficiently large, then $|w|<|z^d|.$
\end{lemma}
\begin{proof}
If  $|w|>|z|^{d}$ for arbitrarily large $|z|$ and $d$, then  $|w|^n|q_n(z)|$ dominates the other terms in the form (2), so $H(z,w)$ cannot vanish.
\end{proof}

\begin{proof}[Proof of Theorem \ref{thm:invariantcurve}]
By Lemma \ref{lemma:fastgrowth} there exist $(z_j, w_j)$ with $z_j \rightarrow \infty$,  $H(z_j,w_j)=0$ and $|f(z_j)| > |z_j|^j$.
Let $(z_j^\prime, w_j^\prime) = F(z_j, w_j)$ so that  $z^\prime_j = f(z_j)-\delta w_j$ and $w^\prime_j=z_j$.
Since $\{H=0\}$ is invariant we have that  $H(z^\prime_j,w_j^\prime)=0$.

 By Lemma \ref{secondlemma} there exists $d\in \N$ such that  $|w_j|<|z_j|^d$ for $j$ sufficiently large. Hence for $j$ sufficiently large
\begin{equation}\label{doublyso}
|z_j^\prime| = |f(z_j)-\delta w_j| \geq |f(z_j)|- |\delta w_j| \geq |f(z_j)|/2.
\end{equation}
It follows that
$$
|p(w^\prime_j) (z^\prime_j)^{N_1}| \ge c |z^\prime_j|^{N_1} \ge c |z_j^\prime|^{N_1 -1} \cdot \frac{|z_j|^j}{2},
$$
where $c>0$ is a constant.
But since $z_j = w_j^\prime$ it follows that for large enough $j$, all terms of the form $\alpha_{k,\ell}(z_j^\prime)^k(w_j^\prime)^\ell$ for $k \le N_1 -1$ will be negligible compared to $p(w^\prime_j) (z^\prime_j)^{N_1}$, which contradicts $H(z_j^\prime, w_j^\prime) = 0$.

\end{proof}

\section{Classification of recurrent components}\label{sectionrecurrent}

In this section we only assume that $F$ is a holomorphic automorphism of $\mathbb C^2$ with constant Jacobian $\delta$.
\begin{defn}
A point $x\in \C^2$ is \emph{recurrent} if its orbit $(F^n(x))$ accumulates  at $x$ itself. A periodic Fatou component $\Omega$ is called  \emph{recurrent} if there exists a point $z \in \Omega$ whose orbit $(F^n(z))$ accumulates at a point $w \in \Omega$.
\end{defn}

Since the class of holomorphic automorphism of $\mathbb C^2$ with constant Jacobian is closed under composition,  by replacing $F$ with an iterate we can restrict to the case where $\Omega$ is invariant.
For an invariant Fatou component $\Omega$, a  \emph{limit map} $h$  is a holomorphic function $h:\Omega\ra\P^2$  such that $f^{n_k}\ra h$ uniformly on compact sets of $\Omega$ for some subsequence $n_k\ra\infty$.

\begin{thm}\label{thm:recurrent} Let $F$ be a holomorphic automorphism of $\mathbb C^2$ with constant Jacobian $\delta$ and
let $\Omega$ be an invariant recurrent Fatou component for $F$. Then there exists a holomorphic retraction $\rho$ from $\Omega$ to a closed complex submanifold $\Sigma\subset \Omega$, called the \emph{limit manifold}, such that
for all limit maps $h$ there exists an automorphism $\eta$ of $\Sigma$ such that $h=\eta\circ \rho$. Every orbit converges to $\Sigma$, and $F|_\Sigma\colon \Sigma\to \Sigma$ is an automorphism. Moreover,
\begin{itemize}
\item If ${\rm dim}\,\Sigma=0$, then $\Omega$ is the basin of an attracting fixed point, and is biholomorphically equivalent to $\mathbb C^2$.
\item  If ${\rm dim}\,\Sigma=1$, either $\Sigma$ is  biholomorphic to a circular domain $A$, and there exists a biholomorphism from $\Omega$ to $A\times \C$ which conjugates the map $F$ to
$$(z,w)\mapsto (e^{i\theta}z, \frac{\delta}{e^{i\theta}} w),$$
or there exists $j\in \N$ such that $F^j|_\Sigma={\rm id}_\Sigma$, and
  there exists a biholomorphism from $\Omega$ to $\Sigma\times \C$ which conjugates the map $F^j$ to
$$(z,w)\mapsto (z, \delta^j w).$$
\item   ${\rm dim}\,\Sigma=2$  if and only if $|\delta| = 1$. In this case there exists a sequence of iterates converging to the identity on $\Omega$.
\end{itemize}
\end{thm}

By a circular domain we mean either the disk, the punctured disk, an annulus, the complex plane or the punctured plane.
For the  polynomial H\'enon maps case, see \cite{BS1991} and  \cite{FSrecurrent}.


Let $(F^{n_k})$ be a convergent subsequence of iterates on $\Omega$, with $F^{n_k}(z) \rightarrow w\in \Omega$. We denote the limit of $(F^{n_k})$ by $g$.

\begin{lemma}
The image $g(\Omega)$ is contained in $\mathbb C^2.$
\end{lemma}
\begin{proof}
If there is a point $x\in \Omega$ for which $F(x)$ belongs to the line at infinity $\ell^\infty$, then $g(\Omega)\subset \ell^\infty$ (see e.g. the proof of Lemma \ref{projectiveimpliesintrinsic}), which gives a contradiction.
\end{proof}

\begin{defn}
We define the  \emph{maximal rank} of $g$ as   $\max_{z\in \Omega}{\rm rk}( d_zg)$.
\end{defn}
\subsection{Maximal rank 0}
\begin{lemma}
Suppose that $g$ has maximal rank $0$. Then $g(\Omega)$ is the single point $w$, which is an attracting fixed point.
\end{lemma}
\begin{proof}
Since the maximal rank is $0$, the map $g$ is constant and must therefore equal $w$. Since $F$ and $g$ commute, the point $w$ must be fixed.
Suppose that the differential $d_wF$ has an eigenvalue of absolute value $\geq 1.$ Then the same is true for all iterates
$F^{n_k}$. Hence they cannot converge to a constant map. So $w$ must be an attracting fixed point.
\end{proof}

It follows that $\Omega$ is the attracting basin of the point $w$, and the entire sequence $F^n$ converges to $g$.
 In this case the limit manifold $\Sigma$ is the point $\{w\}$.

\subsection{Maximal rank 2}

\begin{thm}
Suppose that $g$ has maximal rank $2$. Then there exists a subsequence
$(m_k)$ so that $F^{m_k}\rightarrow \mbox{Id}$ on $\Omega.$
\end{thm}

\begin{proof}
Let $x$ be a point of maximal rank 2. There exist an open  neighborhood $U$ of $x$  and an open  neighborhood $V'$ of $g(x)$ such that $g\colon U\to V'$ is a biholomorphism.  Denote $h:=g^{-1}$ defined on $V'$. Let $V\subset \subset V'$ be an open neighborhood of  $g(x)$.
Since $F^{n_k}\rightarrow g$ on $U$,   we have that $V\subset F^{n_k}(U)$ for large $k$ and the maps $(F^{n_k})^{-1}$ converge to $h$ uniformly on compact subsets of $V$. In particular $V\subset \Omega$.
We can then write $F^{n_{k+1}-n_k}=F^{n_{k+1}}\circ (F^{n_k})^{-1}$ on $V$. If we set $m_k:=n_{k+1}-n_k$, then $F^{m_k}\rightarrow \mathrm{Id}$ on $V$. Since we are in the Fatou set this implies that
$F^{m_k}\rightarrow \mathrm{Id}$ on $\Omega.$
\end{proof}

It follows that every point $p\in \Omega$ is recurrent and that $F$ is volume preserving. The following fact is trivial but we recall it for convenience.
\begin{lemma}\label{hurwitzlimit} Let $(G_n\colon\Omega\subset \C^2\to \C^2)$ be a sequence of injective holomorphic mappings which are volume preserving. If $G_n$ converges to $G$ uniformly on compact subsets, then $G$ is holomorphic, injective and volume preserving.
\end{lemma}
\begin{proof} The map $G$ is holomorphic and $dG_n\stackrel{n\to\infty}\longrightarrow dG$, and thus $G$ is volume preserving.  Thus by Hurwitz Theorem  $G$ is  injective.
\end{proof}

\begin{prop}
If $g$ has maximal rank $2$ then each orbit $(F^n(z))$ is contained in a compact subset of $\Omega$.
\end{prop}
\begin{proof}
Let $(K_j)$ be an exhaustion of $\Omega$ by compact subsets such that   $K_j \subset \stackrel{\circ} K_{j+1} $ for all $j\in \N$.
By passing to a subsequence of the exhaustion if needed, we may assume that $F(K_j) \subset \subset K_{j+1}$. Let $p\in \Omega$. We can assume that $p\in K_1$, and  let  $r>0$ be such that $B(p,r)\subset K_1$.

We may assume that if $F^n(p) \in K_j$ then $F^n(B(p,r)) \subset K_{j+1}$.
Indeed, suppose by contradiction that there exist $j\in \N$ and subsequence $\ell_k$ such that $F^{\ell_k}(p)\in K_j$ for all $k\in \N$, but
\begin{equation}\label{hopecontradicts}
F^{\ell_k}(B(p,r))\not\subset   K_{j+k}.
\end{equation}
Then up to passing to a subsequence, $F^{\ell_k}(p)$ converges uniformly on $B(p,r)$ to a holomorphic map $h$ such that
$h(B(p,r))$ is open and does not intersect the line at infinity $\ell^\infty$. Thus $h(B(p,r))\subset \subset \Omega$, which contradicts (\ref{hopecontradicts}).
Similarly we may assume that if $F^n(p) \notin K_j$ then $F^n(B(p,r)) \cap K_{j-1} = \varnothing$.

Suppose by contradiction that the orbit of $p$ is not contained  in a compact subset of $\Omega$. Then the orbit of $p$ is not contained in any  $K_j$. But since $p$ is a recurrent point, the orbit of $p$ must also return to $K_1$ infinitely often. Thus, there exists a sequence $k_1 < l_1 < m_1 < k_2 < l_2 < m_2 < \ldots$ and a strictly increasing sequence $(n_j)$, $n_j\geq 3$, such that
\begin{enumerate}
\item[(i)] Each $F^{k_j}(p)$ lies in $K_4 \setminus K_3$
\item[(ii)] For $k_j < n < l_j$ the points $F^n(p)$ lie outside of $K_3$.
\item[(iii)] Each $F^{l_j}(p)$ lies outside $K_{n_j}$,
\item[(iv)] Each point $F^{m_j}(p)$ lies in $K_1$,
\item[(v)] For $k_j < n < m_j$ the points $F^n(p)$ lie in $K_{n_{j+1}-2}$.
\end{enumerate}
We claim that the sets $F^{k_j}(B(p,r))$ must be pairwise disjoint. To see this, suppose that $F^{k_i}(B(p,r)) \cap F^{k_j}(B(p,r)) \neq \varnothing$ for some $i < j$. Then clearly $$F^{k_i+n}(B(p,r)) \cap F^{k_j+n}(B(p,r)) \neq \varnothing,\quad \forall n \in \mathbb N.$$
If $l_j -k_j> m_i-k_i$, then a contradiction is obtained since $F^{k_j+m_i - k_i}(B(p,r)) \cap K_{2} = \varnothing$ due to (ii) while $F^{m_i}(B(p,r)) \subset K_2$ due to (iv).
If $l_j-k_j < m_i-k_i$, then a contradiction is obtained since $F^{l_j}(B(p,r)) \cap K_{n_{j}-1} = \varnothing$ due to (iii), while $F^{k_i + l_j - k_j}(B(p,r)) \subset K_{n_{i+1} -1}\subset K_{{n_j}-1}$ due to (v). Finally, if $l_j -k_j= m_i-k_i$,  then $F^{m_i}(B(p,r))\cap F^{\ell_j}(B(p,r))\neq \varnothing$, which contradicts (iii) and (iv).
This proves the claim.
Since $F$ is volume preserving and the  volume of $K_4 \setminus K_3$ is finite, we have a contradiction.

\end{proof}

\begin{corollary}
The limit of any convergent subsequence $(F^{n_k})$ is an automorphism of $\Omega$.
\end{corollary}
In this case the limit manifold $\Sigma$ is the whole $\Omega$.
\begin{remark}\label{independentrank}
It follows that the maximal  rank of a limit map is independent of the chosen convergent subsequence.
\end{remark}

\subsection{Maximal rank 1}\label{rankone} We now consider the case where the limit map $g$ has maximal rank $1$. By Remark \ref{independentrank}  every other limit of a convergent subsequence on $\Omega$ must also have maximal rank $1$.

Recall that  $(F^{n_k})$ is a convergent subsequence of iterates on $\Omega$ such that $F^{n_k}(z) \rightarrow w\in \Omega$. Replace $n_k$ by a  subsequence so that $n_{k+1}-n_k \nearrow \infty.$ Let $(m_k)$ be a subsequence of $(n_{k+1}-n_k)$ so that $F^{m_k}$ converges, uniformly on compact subsets of $\Omega$. From now on we assume that $g$ is the limit of the sequence $(m_k)$.
\begin{remark}\label{needrecurrent}
Notice that   $g(w)=w$.
Actually, if follows by the construction that there exists an open neighborhood $N$ of $z$ in $\Omega$ such that $g(N)\subset \Omega$ and for all $y\in g(N)$, $g(y)=y$.
\end{remark}

\begin{lemma}
The map $F$ is strictly volume decreasing.
\end{lemma}
\begin{proof}
By assumption the Jacobian determinant $\delta$ of $F$ is constant. Since $\Omega$ is recurrent we have $|\delta| \le 1$. Since $(F^{n_k})$ converges to the map $g: \Omega \rightarrow \mathbb C^2$ of maximal rank $1$, it follows that $|\delta| < 1$.
\end{proof}

We write $\Sigma:=g(\Omega)$. Notice that $\Sigma$ is a subset of $\overline{\Omega}\cap\mathbb C^2$, and that, since $F$ and $g$ commute, the map $F|_\Sigma\colon \Sigma\to \Sigma$ is bijective. We need a Lemma.

\begin{lemma}\label{isolatedpoint}
Let $U$ be an open set in $\C^2$, and let $h\colon U\to \C^2$ be a holomorphic map of maximal rank $1$.
Then for all $w\in h(U)$,  the fiber $h^{-1}(w)$ has no isolated points.
\end{lemma}

\begin{proof}
Assume by contradiction that $q\in h^{-1}(w)$ is  isolated. If $\epsilon$ is small enough, then $h(\partial B(q,\epsilon))$
is disjoint from $w.$ Hence there exists a small ball
$B(w,\delta)$ which is disjoint from $h(\partial B(q,\epsilon))$ Hence if we restrict
$h$ to $V:=h^{-1}(B(w,\delta))\cap B(q,\epsilon)$, then $h:V\rightarrow B(w,\delta)$ is a proper holomorphic map.
Let $\zeta\in V$ be such that ${\rm rk}_\zeta h=1$. Then the level set $h^{-1}(h(\zeta))$ contains a closed analytic curve in $V$. Such curve is not relatively compact in $V$, and this contradicts properness.
\end{proof}

It is not clear a priori that $\Sigma$ is a complex submanifold, but we will show, following a classical normalization procedure, that there exists a smooth Riemann surface $\widehat\Sigma$ such that the self-map $F$ on $\Sigma$ can be lifted to a holomorphic automorphism $\widehat F$ on $\widehat \Sigma$.  Note that such normalization  procedure  was used in a similar context in \cite{weickert}.

\begin{lemma}\label{injectivedisk}
For each point $z \in \Omega$ there is an open connected neighborhood $U(z)\subset \Omega$, an affine disk $\Delta_z\subset \Omega$ through $z$ and an injective holomorphic mapping $\gamma_z\colon \Delta_z\to \C^2$ such that
\begin{enumerate}
\item $\gamma_z(\Delta_z)$ is an irreducible local complex analytic curve which is smooth except possibly at $\gamma_z(0)$ where it could have a  cusp singularity,
\item $\gamma_z(\Delta_z)= g(U(z))$.
\end{enumerate}
Moreover, if $g$ has rank $1$ at $z$, then $\gamma_z(\Delta_z)$ is smooth and  $\gamma_z=g|_{\Delta_z}$.
\end{lemma}

\begin{proof}
If $g$ has rank 1 at $z$, the result follows immediately from the constant rank Theorem. So suppose that $g$ has rank $0$ at $z$, and let $\Delta_z\subset \Omega$ be an affine disk through $z$ on which $g$ is not constant.
By the Puiseux expansion of $g: \Delta_z \rightarrow \mathbb C^2$, it follows that, up to taking a smaller $\Delta_z$, $g(\Delta_z)$ is an irreducible local complex analytic curve (with possibly an isolated cusp singularity at $g(z)$). Hence  $g(\Delta_z)$   is the zero set of a holomorphic function $G$ defined in a open neighborhood $V$ of $g(z)$.
Let  $U(z)$ be the connected component of $g^{-1}(V)$ containing $z$.
We claim that $G\circ g$ vanishes identically on $U(z)$, which implies that $g(U(z))=g(\Delta_z)$. If not, then $(G\circ g)^{-1}(0)$ is a closed complex analytic curve in $U(z)$ containing $\Delta_z$. Pick a point $q\in \Delta_z$ where  locally  $(G\circ g)^{-1}(0)=\Delta_z$. Then $g^{-1}(g(q))$ is isolated at $q$ since $g$ is not constant on $\Delta_z$, which gives a contradiction by Lemma \ref{isolatedpoint}.

Finally, again by  the Puiseux expansion of $g: \Delta_z \rightarrow \mathbb C^2$, there exists a holomorphic injective map $\gamma_z\colon \Delta_z\to \C^2$ such that
$\gamma_z(\Delta_z)=g(\Delta_z)$.

\end{proof}


%
%
\begin{remark}
For all $z\in \Omega$, there exists a unique surjective holomorphic map $h_z\colon U(z)\to \Delta_z$ such that  $g =\gamma_z\circ h_z$  on the neighborhood $U(z)$. If  if $g$ has rank 1 at $z$, then $h_z|_{\Delta_z}={\rm id}$.

\end{remark}

Consider the disjoint union $\coprod_{z\in \Omega} \Delta_z$, and define an equivalence relation in the following way:
 $(x,z)\simeq (y, w)$ if and only if   $\gamma_z(x)=\gamma_w(y)$ and the images  coincide locally near this point.
Define $\widehat \Sigma$ as  $\coprod_{z\in \Omega} \Delta_z$, endowed with the quotient topology, and  denote $\pi_{\simeq}\colon \coprod_{z\in \Omega} \Delta_z\to \widehat \Sigma$ the projection to the quotient.  It is easy to see that the map $\pi_{\simeq}$ is open. For all $z\in \Omega$, define a homeomorphism $\pi_z\colon \Delta_z\to \widehat \Sigma$ as $ \pi_z(x):=[(x,z)]$.

\begin{defn}\label{defgamma}
We define a continuous map $\gamma\colon \widehat \Sigma\to \C^2$ such that $\gamma(\widehat \Sigma)= \Sigma$ in the following way:
$\gamma([(x,z)])=\gamma_z(x).$
Notice that this is well defined.
The map $g: \Omega \rightarrow \C^2$ can be lifted to a unique  surjective continuous map $\widehat{g}: \Omega \rightarrow \widehat{\Sigma}$ such that $g=\gamma\circ \widehat g.$ Such map is defined on $U(z)$ as $\widehat g:= \pi_z \circ h_z.$ Notice that if $g$ has rank 1 at $z$ then $\widehat g|_{\Delta_z}=\pi_z$.
\end{defn}

\begin{lemma}
The topological space $\widehat \Sigma$ is  connected, second countable and Hausdorff.
  \end{lemma}
\begin{proof}
Since $\widehat \Sigma=\hat{g}(\Omega)$, and $\Omega$ is connected, it follows that $\widehat \Sigma$ is connected. Since $\widehat g$ is open, it follows also that $\widehat\Sigma$ is second countable. Let $[(x,z)]\neq [(y,w)]\in \widehat \Sigma$. Then we have two cases. Either $\gamma_z(x)\neq \gamma_w(y)$, or $\gamma_z(x)= \gamma_w(y)$ but the images do not  coincide locally near this point.
  In both cases there  exist a  neighborhood $U\subset \Delta_z$ of $x$ and a neighborhood $V\subset \Delta_w$ of $y$ such that   $\pi_{\simeq}(U)\cap \pi_{\simeq}(V)=\varnothing.$

\end{proof}

 We claim that the collection of charts $(\pi_z)_{z\in \Omega}$ gives  $\widehat \Sigma$ the structure of a smooth Riemann surface.
 Let $z,w\in \Omega$ such that $ \pi_z(\Delta_z)\cap \pi_w(\Delta_w)\neq \varnothing.$
 Then consider the map $$\pi_w^{-1}\circ \pi_z\colon \pi_z^{-1}(\pi_z(\Delta_z)\cap \pi_w(\Delta_w))\to \pi_w^{-1}(\pi_z(\Delta_z)\cap \pi_w(\Delta_w)). $$
  Let $x\in \Delta_z$, $y\in \Delta_w$ such that $\pi_z(x)=\pi_w(y)$. This means that $\gamma_z(x)=\gamma_w(y)$ and the images coincide locally near this point. There exists an open neighborhood $U\subset \Delta_z$ of $x$, an open neighborhood $V\subset \Delta_w$ of $y$,  and a  unique biholomorphic function $k\colon U\to V$ such that
 $\gamma_w\circ k=\gamma_z$. It is easy to see that $k=\pi_w^{-1}\circ \pi_z$ on $U$.

\begin{remark}\label{global-lift2}
With the complex structure just defined on $\widehat \Sigma$, the maps $\gamma$ and  $\widehat g$ are holomorphic.
\end{remark}

\begin{defn}
Define $R\subset \widehat \Sigma$ as the set of points $\zeta\in \widehat \Sigma$ such that there exists $z\in \Omega$ with $\widehat g(z)=\zeta$ and ${\rm rk}_z\widehat g=1$.
\end{defn}
\begin{lemma}
The set $\widehat \Sigma\setminus R$ is discrete.
\end{lemma}
\begin{proof}
Let $w\in \Omega$ such that ${\rm rk}_w\widehat g=0$.
By the identity principle there exists a neighborhood $V$ of $w$ in $\Delta_w$ such that  ${\rm rk}_z(h_w|_{\Delta_w})=1$ for all $z\in V\setminus\{w\}$.
The result follows since   $\widehat g= \pi_w\circ h_w$ on $U(w)$, and $\pi_w\colon \Delta_w\to \pi_w(\Delta_w)$ is a biholomorphism.

\end{proof}

\begin{lemma}
There exists a unique holomorphic map $\widehat{F}: \widehat{\Sigma}\rightarrow \widehat{\Sigma}$ such that the following diagram commutes:

\SelectTips{xy}{12}
\[ \xymatrix{\Omega \ar[rrr]^g\ar[rrd]^{\widehat g}\ar[dd]^F &&& \Sigma \ar[dd]^F\\
&& \widehat\Sigma \ar[ru]^\gamma \ar[dd]^(.25){\widehat F}\\
\Omega\ar'[rr]^g[rrr] \ar[rrd]^{\widehat g} &&& \Sigma\\
&& \widehat\Sigma \ar[ru]^\gamma.}
\]

\end{lemma}
\begin{proof}
Let $\zeta\in R$, and let $z\in \Omega$ such that  $\widehat g(z)=\zeta$ and ${\rm rk}_z\widehat g=1$.  Define on $\pi_z(\Delta_z)$ the map $\widehat F:= \widehat g \circ F  \circ \pi_z^{-1}$. This  is well-defined and holomorphic away from the discrete closed set $\widehat \Sigma\setminus R$, and can be extended holomorphically to the whole $\widehat \Sigma$.
\end{proof}
The inverse of $\widehat{F}$ is given by $\widehat{F^{-1}}$, therefore $\widehat{F}$ is an automorphism.

\begin{lemma}\label{openidentity}
The Riemann surface $\widehat \Sigma$ contains an  open subset  on which the sequence $(\widehat{F}^{m_k})$ converges to the identity.
\end{lemma}
\begin{proof}
By Remark \ref{needrecurrent} there exists an open neighborhood $N$ of $z$  such that $g(N)\subset \Omega$ and for all $y\in g(N)$, $F^{m_k}(y)\to y$.
Then, for all $y\in g(N)$, $$\widehat F^{m_k} (\widehat{g}(y))=\widehat g(F^{m_k} (y))\to  \widehat{g}(y).$$
  The set $\widehat g(N)$ is open.
 \end{proof}

\begin{lemma}\label{dichotomy}
Either there exists a $j \in \mathbb N$ for which $\widehat{F}^j = \mathrm{Id}$, or $\widehat{\Sigma}$ is biholomorphic to a circular domain, and the action of $\widehat{F}$ is conjugate to an irrational rotation.
\end{lemma}
Recall that by a circular domain we mean either the  disk, the punctured disk, an annulus, the complex plane or the punctured plane.
\begin{proof}
Assume that  $\widehat{F}^j \neq  \mathrm{Id}$ for all $j\leq 1$.
Since the holomorphic map $\gamma\colon\widehat\Sigma\to \C^2$ is nonconstant, it follows that the Riemann surface $\widehat \Sigma$ is not compact.
Thus if  $\widehat \Sigma$ is not a hyperbolic Riemann surface, then it has to be biholomorphic either
to $\C$ or to  $\C^*$. In both  cases, since $\widehat F$ is an automorphism, it is easy to see that Lemma \ref{openidentity} implies that  $\widehat F$ is a rotation.
If $\widehat \Sigma$ is  hyperbolic, then the family $(\widehat F^n)$ is normal, and thus Lemma \ref{openidentity} implies that the sequence $(\widehat F^{m_k} )$ converges to the identity uniformly on compact subsets of $\widehat \Sigma$. Thus the  automorphism group of $\widehat{\Sigma}$ is non-discrete. Hence (see e.g. \cite[p. 294]{FK}) the Riemann surface $\widehat{\Sigma}$ is biholomorphic to a circular domain and the action of $\widehat{F}$ is conjugate to an irrational rotation.
\end{proof}

\begin{defn}
Define the set $I\subset \widehat \Sigma\times \widehat\Sigma $ as the set of pairs $(x,y) $ such that  $x\neq y$ and $\gamma(x)=\gamma(y)$.
Define the set $C\subset \widehat\Sigma$  as the set of $x$ such that $\gamma\colon \widehat \Sigma\to \C^2$ has rank $0$ at $x$.
 \end{defn}

Since the map $\pi_\simeq$ is open, it follows immediately that the set $I$  is discrete in $\widehat \Sigma\times \widehat \Sigma$ and that the set $C$  is discrete in $\widehat\Sigma$.

Our goal is to prove that the set $\Sigma$ is a closed complex submanifold of $\Omega$. We will first consider the case where $\widehat{\Sigma}$ is biholomorphic to a circular domain  and $\widehat{F}$ is conjugate to an irrational rotation.

\begin{lemma}
If $\widehat F$ is conjugate to an irrational rotation then there is at most one element  $\zeta_0 \in C$. The set $I$ is empty, and thus the map $\gamma\colon \widehat\Sigma\to \Sigma$ is injective.
\end{lemma}
\begin{proof}
The set $C$ is invariant by $\widehat F$. Since the action of $\widehat{F}$ is conjugate to an irrational rotation, and since $C$ is discrete it follows that  $C$ can only contain the  center of rotation $\zeta_0$ (if there is one).

Similarly, the set $I$ is  invariant by the map $(x,y)\mapsto (\widehat F(x),\widehat F(y)),$ but this contradicts the discreteness of $I$.
\end{proof}

Let $\widehat{V}\subset  \subset \widehat \Sigma\setminus \{\zeta_0\}$ be open and invariant under the action of $\widehat{F}$, and let $V$ be its image in $\mathbb C^2$, which  is an embedded complex submanifold of $\C^2$.

We claim that there exists a continuous function $\varphi: V \rightarrow \mathbb (0, \infty)$, bounded from above and from below by compactness, such that
\begin{equation}\label{cocycle}
\|d_z(F|_V)\| = \frac{\varphi(z)}{\varphi(F(z))}.
\end{equation}

Indeed, regardless of whether $\widehat{\Sigma}$ is a hyperbolic or Euclidean Riemann surface, there exists a conformal metric $\|\cdot \|$ on $\widehat{\Sigma}$ which is invariant under $\widehat{F}$.
The function $$\varphi(z):=\|d_{z}\gamma^{-1}\|$$ satisfies (\ref{cocycle}).

Given $\epsilon>0$ and $z \in V$, we define the tangent cone $\mathcal C_z \subset T_z(\mathbb C^2)$ by
$$
\mathcal{C}_z = \{ w : |\langle w, v_z^\bot\rangle | \le \epsilon \varphi(z)^2 |\langle w, v_z\rangle |\},
$$
where  $v_z$ is a unit tangent vector to $z \in V$ and  $v_z^\bot$ is a unit vector orthogonal to $v_z$.

\begin{lemma}
One can choose $\epsilon >0$ sufficiently small so that
$$
d_zF (\mathcal{C}_z) \subset\subset  \mathcal{C}_{F(z)}.
$$
\end{lemma}
\begin{proof}
This is a matter of linear algebra. Without loss of generality we may assume that $v_z = v_{F(z)} = (1,0)$. Thus, the cone field $\mathcal{C}_z$ contains all vectors $w = (w_1, w_2)$ for which
$$
|w_2| \leq \epsilon\varphi(z)^2|w_1|.
$$
The vector $d_zF(w)$ is given by
$$
d_zF(w) = \left(\theta(z)w_1 +\alpha(z) w_2, \beta(z) w_2\right),
$$
where $|\theta(z)|=\frac{\varphi(z)}{\varphi(F(z))}$.
Since $F$ has constant Jacobian determinant $\delta$ it follows that
$$
|\theta(z)\beta(z)| = |\delta|.
$$
Since  $|\alpha(z)|$ is bounded on $V$, by choosing $\epsilon$ sufficiently small we can guarantee that for all  $z \in V$ we have
$$
\frac{1-|\delta|}{ \varphi(F(z))\varphi(z)} >  \epsilon |\alpha(z)|,
$$
 from which it follows that
$$
\epsilon  \varphi(F(z))^2  |d_zF(w)_1| >|d_zF(w)_2|.
$$
Thus, $d_zF$ sends the cone $\mathcal{C}_z$ strictly into $\mathcal{C}_{F(z)}$.
\end{proof}

\begin{lemma}
$\Sigma\setminus \gamma(\zeta_0)$ is contained in $\Omega$.
\end{lemma}
\begin{proof}
Since $F$ is $C^1$, we can extend the invariant cone field to a neighborhood $\mathcal{N}(V )$. Let $z$ be a point whose forward orbit remains in  $\mathcal{N}(V )$, which holds in particular for all points in $V$. Then there exists a  stable manifold $W^s(z)$, transverse to $V$, and  these stable manifolds fill up a neighborhood of $V$, see the reference \cite{HPS1970} for background on normal hyperbolicity.

The forward iterates of F form a normal family on this neighborhood, which implies that $V$ is contained in the Fatou set.
\end{proof}

\begin{lemma}
The set $C\subset \widehat \Sigma$ is empty, and   $\Sigma\subset \Omega$.
\end{lemma}
\begin{proof}
If $\widehat \Sigma$ has no center of rotation, then there is nothing to prove.
Suppose for the purpose of a contradiction that there is a  center of rotation $\zeta_0\in \widehat\Sigma$ and that $\zeta_0\in C$.

Then $\Sigma$ has a cusp at  $z:=\gamma(\zeta_0)$. Notice that $F(z)=z$. Since $\widehat F$ acts on $\widehat{\Sigma}$ as a rotation, it follows that the tangent direction of $\Sigma$ to $z$ is an eigenvector of $d_zF$ with eigenvalue $|\lambda_1|=1$. Since $F$ is strictly volume decreasing, the other eigenvalue $\lambda_2$ of $d_zF$ satisfies $0<|\lambda_2|<1$. Thus we obtain a forward invariant cone in $T_z(\mathbb C^2)$, centered at the line $T_z(\Sigma)$. Extending this cone to a constant cone field in a neighborhood of $z$, it follows that we obtain stable manifolds through $z$ and all nearby points in $\Sigma$, giving a continuous Riemann surface foliation near the point $z$.

Since $\widehat F$ acts on $\widehat\Sigma$ as a rotation, the stable manifolds through different points in $\Sigma$ must be distinct. However, this is not possible in a neighborhood of $z$. To see this, let $h$ be a locally defined holomorphic function such that $\Sigma$ equals the zero set of $h$ near $z$. We may assume that $h$ vanishes to higher order only at $z$. Now consider the restriction of $h$ to the stable manifold through $z$. Then $h$ has a multiple zero at $z$, hence by Rouch\'e's Theorem, the number of zeroes for nearby stable manifolds is also greater than one. But since nearby stable manifolds are transverse to $\Sigma$, and $h$ does not vanish to higher order in nearby points, it follows that the restriction of $h$ to nearby stable manifolds must have multiple single zeroes. Hence these nearby stable manifolds intersect $\Sigma$ in more than one point, giving a contradiction.

We conclude that $C=\varnothing$. As in the proof of the previous lemma it follows that $\Sigma$ is contained in the Fatou set, and thus in $\Omega$.
\end{proof}

\begin{corollary}
The map $g\colon \Omega\to \Sigma$ is a holomorphic retraction. In particular  $\Sigma$ is a closed smooth one-dimensional embedded submanifold.
\end{corollary}
\begin{proof}
On $\Sigma=g(\Omega)\subset \Omega$,
$$g=\lim_{k\to \infty} F^{m_k}={\rm Id},$$ which proves that $g$ is a holomorphic retraction.
\end{proof}

In the case where $\widehat{F}^j$ equals the identity, it follows that $F^j$ equals the identity on $\Sigma$. In this case we immediately get  stable manifolds transverse to $\Sigma$, which imply  as above that $g\colon \Omega\to \Sigma$ is a holomorphic  retraction.

\begin{lemma}
The retraction $g$ has constant rank $1$ on $\Omega$.
\end{lemma}
\begin{proof}
Since $g$ is a retraction  to the 1-dimensional manifold $\Sigma$, there exists a neighborhood $V$ of $\Sigma$ such that ${\rm rk}_x g=1$ for all $x\in V$. Let $x\in\Omega $ and $N\geq 0$ be such that  $y=F^N(x)\in V$. The result follows from the fact that $F^N$ has rank 2 and that
$$g\circ F^N(x)=F^N\circ g(x).$$

\end{proof}

%

\begin{remark}
The stable manifolds of the points in $\Sigma$ fill up a neighborhood of $\Sigma$. Since all orbits in the Fatou component $\Omega$ get close to $\Sigma$, this implies that
the stable manifolds of the points in $\Sigma$ fill up the whole $\Omega$. Thus for any limit map $h\colon \Omega\to \P^2$ we have that
$h(\Omega)\subset \Sigma.$
Moreover, by Lemma \ref{dichotomy}, the restriction $h|_{\Sigma}$ is an automorphism of $\Sigma$, and thus $h(\Omega)=\Sigma$.
Notice that for all $z\in \Sigma$ the fiber $g^{-1}(z)$ coincides with the stable manifold $W^s(z)$. Hence $h=h|_\Sigma\circ g.$
\end{remark}

\begin{remark} In the view of the contents of this section if  a Fatou component $\Omega$
is recurrent then  for every $z\in \Omega$  the orbit is relatively compact in $\Omega$.
\end{remark}

%


Now we investigate the complex structure of $\Omega$.
Notice that $\Sigma$, being an open Riemann surface, is Stein.
\begin{prop}
If some iterate $F^j|_\Sigma$ is the identity, then there exists a biholomorphism  $\Psi\colon \Omega\to \Sigma\times \C$ which conjugates the map $F^j$ to
$(z,w)\mapsto (z, \delta^j w).$
\end{prop}
\begin{proof}
Let $L$ be the holomorphic line bundle on $\Sigma$ given by ${\rm Ker}\, dg$ restricted to $\Sigma$.  Since $\Sigma$ is Stein, there exists a neighborhood $U$ of $\Sigma$ and an injective holomorphic map $h\colon U\to L$ such that  for all $z\in \Sigma$  we have  that $h(z)=0_z$, and $h$ maps the fiber $g^{-1}(z)\cap U$ biholomorphically into a neighborhood of $0_z$ in the fiber $L_z$  \cite[Proposition 3.3.2]{forstnericbook}. We can assume that $d_z h|_{L_z}={\sf id}_{L_z}$.
Notice that the map $dF^j|_{L_z}\colon L_z\to L_z$ acts as multiplication by $\delta^j$.
The sequence $(  dF|_L^{-nj}  \circ h\circ  F^{nj})$  is eventually defined on compact subsets of $\Omega$ and  converges uniformly on compact subsets to a biholomorphism $\Psi\colon \Omega\to L$, conjugating $F^j$ to $dF^j|_L\colon L\to L$.
The line bundle $L$ is holomorphically trivial since $\Sigma$ is a Stein Riemann surface.

\end{proof}

If no iterate of $F$ is the identity on $\Sigma$, then $\Sigma$ is a biholomorphic to a circular domain $A$, and by the same proof as in
\cite[Proposition 6]{BS1991} there exists a biholomorphism from $\Omega$ to $A\times \C$ which conjugates the map $F$ to
$$(z,w)\mapsto (e^{i\theta}z, \frac{\delta}{e^{i\theta}} w).$$

\subsection{Transcendental H\'enon maps}

So far in this section we have not used that the maps we study here are of the form $F(z,w) = (f(z)-\delta w, z)$. Absent of any further assumptions we do not know how to use this special form to obtain a more precise description of recurrent Fatou components. We do however have the following consequence of Proposition \ref{prop:discrete} and Theorem \ref{thm:recurrent}.

\begin{corollary}\label{cor:Wiman}
Let $F$ be a transcendental H\'enon map, and assume that $f$ has order of growth strictly smaller than $\frac{1}{2}$. Then any recurrent periodic Fatou component $\Omega$ of rank $1$ must be the attracting basin of  a Riemann surface $\Sigma \subset \Omega$ which is biholomorphic to a circular domain, and $f$ acts on $\Sigma$ as an irrational rotation.
\end{corollary}

%
%
%


\section{Baker domain}\label{sectionbaker}
A Baker domain for a transcendental function $f$ in $\C$  is a periodic Fatou component on which the iterates converge to the point $\infty$, an essential singularity for $f$. Notice that, in contrast, every polynomial $p$ in $\C$ can be extended to a rational function of $\P^1$ for which $\infty$ is a super-attracting fixed point.
The first example of a Baker domain was given by Fatou \cite{Fatou}, who considered the
function
\begin{equation}\label{bakerexample}
f(z) = z + 1 + e^{-z}
\end{equation}
and showed that the right half-plane $\mathbb{H}$ is contained in an invariant
Baker domain.

For holomorphic automorphisms of $\C^2$ the definition of essential singularity at infinity has to be adapted in order to be meaningful.

\begin{defn}
Let $F$ be a holomorphic automorphism of $\C^2$.
We call a point $[p:q:0]\in \ell^\infty$ an  \emph{ essential singularity at infinity} if for all $[s:t:0] \in \ell^\infty$ there exists  a sequence  $(z_n,w_n)$ of points in $\C^2$ converging to $[p:q:0]$ whose images $F(z_n,w_n)$ converge to $[s:t:0]$.
\end{defn}

\begin{remark}\label{essentialsingularity}
If $F$ is a transcendental H\'enon map, then any point in $\ell_\infty$ is an essential singularity at infinity.
To see this, it is in fact sufficient to consider only points on the line $\{(p\zeta, q\zeta)\}$ when $p \neq 0$. Define the entire function
$$
g(\zeta) = \frac{f(p \zeta) - f(0) - \delta q \zeta}{p \zeta}.
$$
Since $f$ is transcendental, the function $g$ must also be transcendental. Let $\zeta_n$ be a sequence of points for which $g(\zeta_n) \rightarrow \frac{s}{t}$. Then
$$
F(p\zeta_n, q\zeta_n) = (p\zeta_n \cdot g(\zeta_n) + f(0), p \zeta_n) \rightarrow [s:t:0].
$$
\end{remark}


Here we consider the iteration of a map on $\C^2$ analogous to (\ref{bakerexample}), namely the   transcendental
 H\'enon map
$$
F(z,w) := (e^{-z} + 2z - w, z),
$$
and we show that it admits an invariant Fatou component $U$ on which the iterates tend to the point  $[1:1:0]$  on the line at infinity.
\begin{remark}
Notice that by Remark \ref{essentialsingularity} the point $[1:1:0]$ is an essential singularity at infinity for the map $F$, and this implies that $F$ cannot be extended, even continuously, to the point $[1:1:0]$.
The situation is radically different   for a polynomial H\'enon map $H$, for which  the  \emph{ escaping set} $$I_\infty:=\{(z,w)\in \C^2\colon \|H^n(z,w)\|\to \infty\}$$ is a Fatou component on which all orbits converge to the point $[1:0:0]$, and the map $H$  extends to  holomorphic  self-map $\widehat H$ of $ \P^2\setminus\{[0:1:0]\}$ with a super-attracting fixed point at $[1:0:0]$.
\end{remark}

For a  transcendental function $f$ in $\C$ it is known that Baker domains are simply connected (proper) domains of $\C$ and  by results of Cowen \cite{Cowen} (see also \cite[Lemma 2.1]{berg}) the function $f$ is semi-conjugate to one of the following automorphisms:
\begin{enumerate}
\item  $z\mapsto \lambda z \in {\rm Aut}(\mathbb{H})$, where $\lambda>1$,
\item $ z\mapsto z\pm i \in {\rm Aut}(\mathbb{H})$,
\item $z\mapsto z+1 \in {\rm Aut}(\C)$.
\end{enumerate}
In our case we show that on the Fatou component $U$ the map $F$ is conjugate to the linear map $L\in {\rm Aut}(\{{\rm Re} (z-w)>0\})$ given by
$$
L(z,w) := (2z - w, z).
$$
We show that the  conjugacy maps $U$ onto the domain $\{{\rm Re} (z-w)>0\}$, proving that $U$ is biholomorphic to $\mathbb{H}\times \C$.

We begin by constructing an appropriate forward invariant domain $R$.
For each $\alpha>0$ we define the domain
$$
R_{\alpha} := \{(z,w) \, :    \Re z>\Re w + \alpha + \eta_\alpha(\Re w)\},
$$
where
$$
\eta_\alpha(x) := \frac{e^{-\alpha}}{1-e^{-\alpha}} \cdot e^{-x}:=A_\alpha \cdot e^{-x}.
$$
Notice that $\eta_\alpha>0$,  $\frac{A_\alpha}{1+A_\alpha}=e^{-\alpha}$, and that the domains $R_\alpha$ are not nested.

\begin{lemma}
Each domain $R_{\alpha}$ is forward invariant.
\end{lemma}

\begin{proof}
Let $(z_0,w_0)\in R_\alpha$.
We claim that $\Re z_1>\Re w_1 + \alpha + \eta_\alpha(\Re w_1)$.
Since $w_1=z_0$,  $$\Re z_1-\Re w_1=\Re e^{-z_0}+\Re z_0-\Re w_0>\alpha+\eta_\alpha(\Re w_0)+\Re (e^{-z_0})\geq \alpha+\eta_\alpha(\Re w_0)-e^{-\Re z_0}.$$
Hence the claim follows if we show that $$\eta_\alpha(\Re w_0)\geq \eta_\alpha(\Re z_0)+ e^{-\Re z_0}.$$
This is the same as $A_\alpha e^{-\Re w_0}\geq  (1+A_\alpha)e^{-\Re z_0}$, or equivalently $e^{-\alpha}\geq e^{-\Re z_0+\Re w_0}$. The latter is satisfied because $\Re z_0-\Re w_0\geq \alpha$ (since $\eta_\alpha$ is always positive).
\end{proof}

We immediately obtain the following.

\begin{corollary}
The domain $$
R := \bigcup_{\alpha>0} R_\alpha
$$ is forward invariant.
\end{corollary}

\begin{remark}\label{driftremark}
Let $(z_0,w_0)\in R_\alpha$. Since $w_1=z_0$ we have that
\begin{equation}\label{drift}
\Re z_1-\Re z_0> \alpha+\eta_\alpha(\Re z_0).
\end{equation}
It easily follows that then $\Re z_n>\Re z_0 +n\alpha$, and that $\Re w_n>\Re z_0 +(n-1)\alpha\ra\infty$.
\end{remark}

\begin{lemma}
All the orbits in $R$ converge uniformly on compact subsets to the point $[1:1:0]$ on the line at infinity.
\end{lemma}
\begin{proof}
Let $K$ be a compact subset of $R_\alpha$ for some $\alpha>0$. Let $M>0$. Then by Remark \ref{driftremark} we have that $|z_n|$ and $|w_n|$ converge to $\infty$ uniformly on $K$.
Moreover, since for all $n\geq 0$ we have that  $z_{n+1}-w_{n+1}=z_n-w_n+e^{-z_n},$ it follows that
$$z_n-w_n= z_0-w_0 +\sum_{j=0}^{n-1} e^{-z_j},$$ and thus $\frac{z_n}{w_n}$ converges to 1 uniformly on $K$.

\end{proof}

The domain $R$ is therefore contained in an invariant Fatou component $U$. Our next goal is to show that $R$ is an absorbing domain for $U$, i.e.
$$
U = A := \bigcup_{n \in \mathbb N} F^{-n}(R).
$$
It is immediate that $A$ is contained in $U$. In order to show that $U$ is not larger than $A$, we will for the first time use the plurisubhamonic method referred to in the introduction.

\begin{defn}
Define the sequence of pluriharmonic functions $(u_n\colon U\to \R)_{n\geq 1}$ by
$$
u_n(z_0,w_0) := \frac{-\Re(z_n)}{n}.
$$
\end{defn}

\begin{lemma}
The functions $u_n$ are  uniformly bounded from above on compact subsets of $U$, and
$$
\limsup_{n\to +\infty} u_n\leq 0.
$$

\end{lemma}
\begin{proof}

Let $K$ be a compact subset of $U$. Since $U$ is a Fatou component, on $K$ the orbits converge uniformly to $[1:1:0]$.
So for every $\epsilon>0$ there exists $n_\epsilon\geq 0$ such that
\begin{equation}
\left|\frac{z_n}{z_{n-1}}\right|=\left|\frac{z_n}{w_n}\right|<1+\epsilon,
\end{equation}
for all $n>n_\epsilon$ and for all $(z_0,w_0)\in K$.
It follows that for every $\epsilon>0$ there exists $C = C_\epsilon$ such that
\begin{equation}\label{zwartepiet}
{\rm max}\{{|z_n|,|w_n|}\} \le C \cdot (1+\epsilon)^n
\end{equation}
for  every $n \in \mathbb N$ and for all   $(z_0,w_0)\in K$. Let $\beta >0$. We claim that there exists an $N>0$ so that $u_n < \beta$ on $K$ for every $n \ge N$. To prove this claim, let us suppose by contradiction that
there exist a sequence $(z_0^k,w_0^k)_{k\in \N}$ in $K$ and a strictly increasing sequence $(n(k))_{k\in\N}$ in $\N$ such that
$$u_{n(k)}(z^k_0, w^k_0) \ge \beta$$
for all $k \in \mathbb N$. Then $\mathrm{Re}(z^k_{n(k)}) \le - n(k) \cdot \beta$. It follows that
$$
\begin{aligned}
|z^k_{n(k)+1}| & = |e^{-z^k_{n(k)}} + 2 z^k_{n(k)} - w^k_{n(k)}|\\
 & \ge |e^{-z^k_{n(k)}}| - |2z^k_{n(k)} - w^k_{n(k)}|\\
 & \ge e^{n(k) \cdot \beta} - 3C(1+\epsilon)^{n(k)}.
\end{aligned}
$$
By taking $\epsilon>0$ sufficiently small so that $e^\beta > (1+\epsilon)$, we obtain, for $k$ sufficiently large,
$$
|z^k_{n(k)+1}| > C (1+\epsilon)^{n(k)+1},
$$
 which contradicts (\ref{zwartepiet}).

The claim implies that there is a uniform bound from above for the functions $u_n$ on $K$, and that
$$
\limsup_{n \rightarrow \infty} u_n \le 0,
$$
which completes the proof.
\end{proof}

\begin{lemma}\label{lemma:negative}
Let $H$ be a compact subset of $A$. Then there exists $\gamma>0$ such that on $H$,
$$\limsup_{n\to +\infty} u_n\leq -\gamma.$$
\end{lemma}

\begin{proof}
Let $K$ be a compact subset of $R_\alpha$ for some $\alpha>0$. Then by Remark \ref{driftremark} we have,
$$
u_n(z_0,w_0)<-\frac{\Re z_0}{n}-\alpha, \quad \forall\, (z_0,w_0)\in K, n\geq 1,
$$
which implies that $\limsup_{n\to +\infty} u_n(z_0,w_0)\leq -\alpha$ for all $(z_0,w_0)\in K$.

Let now $H$ be a compact subset in $A$. Then there exist $\alpha_1,\dots,\alpha_k>0$ and $n_1,\dots, n_k\in \N$ such that
$$H\subset F^{-n_1}(R_{\alpha_1})\cup\dots \cup  F^{-n_k}(R_{\alpha_k}).$$
Thus, on $H$, $$\limsup_{n\to +\infty} u_n\leq {\rm max}\{-\alpha_1,\dots ,-\alpha_k\}.$$

\end{proof}

\begin{lemma}\label{lemma:zero}
On $U\setminus A$, $
\limsup_{n\to +\infty} u_n = 0.
$
\end{lemma}
\begin{proof}
Let $(z_0,w_0)\in U\setminus A$.
Suppose by contradiction that there exists $\alpha>0$ and $N\in \N$ such that
$u_n(z_0,w_0)\leq -\alpha$ for all $n\geq N$, that is,
$$
\frac{\Re z_n}{n}\geq \alpha, \quad \forall\, n\geq N.
$$
Let $n_0\geq 0$ be such that $\eta_{\alpha/3}(\Re z_n)< \alpha/3$ for all $n>n_0$.
For all $0<\beta<\alpha$ there are arbitrarily large $n$ such that  $\Re z_{n+1}-\Re z_n\geq\beta.$
Setting $\beta:= 2\alpha/3$, there exists $n\geq n_0$  such that $$\Re z_{n+1}-\Re z_n> 2\alpha/3>\alpha/3+\eta_{\alpha/3} (\Re z_n),$$ which implies that $(z_{n+1},w_{n+1})\in R_{\alpha/3}$.

\end{proof}

\begin{prop}\label{absorbing}
The set $R$ is an absorbing domain.
\end{prop}
\begin{proof}
Let us assume, for the purpose of a contradiction, that $U \neq A$.  Define
$$
u(z) := \limsup_{n\to+\infty} u_n(z)
$$
and let $u^\star$ be its upper semicontinuous regularization. Then  by \cite[Prop 2.9.17]{Klimek} the function $u^\star$ is plurisubharmonic.
By Lemma \ref{lemma:negative} and  \ref{lemma:zero}  the function $u^\star$ is strictly negative on $A$, and constantly equal to zero on  $U\setminus A$.
But then $u^\star$ contradicts the sub-mean value property at  boundary points $\zeta\in \partial A$.
\end{proof}

\begin{remark}
It is easy to see that for all $n\in \Z$, $(z,w)\in \C^2$,
$$L^n(z,w)=((n+1)z-nw, nz-(n-1)w).$$
\end{remark}

%
%

\begin{defn}
We denote by $\Omega$ the domain $\{(z,w)\in\C^2: \Re({z}- {w})>0\}$, which is biholomorphic to $\mathbb{H}\times \C$.
\end{defn}

\begin{thm}
There exists a biholomorphism $\psi\colon U\to \Omega$ which conjugates $F$ to the map $L$. In particular $U$ is biholomorphic to $\mathbb{H}\times \C$.
\end{thm}
\begin{proof}
We will construct the map $\psi$ as the uniform limit on compact subsets of $U$ of the maps
$$
\psi_n:=L^{-n} \circ F^n\colon U\to \C^2.
$$
Notice that for all $n>1$, the mapping $\psi_n$ is an injective volume-preserving holomorphic mapping.
We have
$$
\begin{aligned}
\delta_n(z_0,w_0) & :=\|L^{-n-1} (F^{n+1}(z_0,w_0))-L^{-n} ( F^n(z_0,w_0))\|\\
& = \|L^{-n}\left(L^{-1}F(z_n,w_n) - (z_n,w_n)\right)\|\\
& = \|L^{-n} (0, -e^{-z_n})\| \\
& = \|(-n e^{-z_n}, (-n-1)e^{-z_n})\|\leq \sqrt 2 (n+1)e^{-\Re z_n}.
\end{aligned}
$$
By Remark \ref{driftremark}, on $R_\alpha$ $$\sqrt 2 (n+1)e^{-\Re z_n}\leq \sqrt 2 (n+1)e^{-\Re z_0-n\alpha},$$
Hence for all $(z_0,w_0)\in R_\alpha$,  we have $\sum_{k=0}^\infty  \delta_n(z_0,w_0)<+\infty,$ and thus
the sequence $(\psi_n)_{n\geq 0}$ converges uniformly on compact subsets of $R$ to an injective volume-preserving (see Lemma \ref{hurwitzlimit}) holomorphic mapping $\psi\colon R\to \C^2$,
satisfying
\begin{equation}\label{conj}
L\circ  \psi=\psi\circ F.
\end{equation}
 Since $R$ is an absorbing domain, $\psi$ extends to an injective volume-preserving holomorphic mapping (still denoted by $\psi$) defined on the whole Fatou component $U$ and  still satisfying (\ref{conj}).
We claim that $\psi(U)=\Omega$.
%

We first show that $\psi(U)\subset \Omega$.
Let $(z_0,w_0)$ be in $R_\alpha$.  Notice that  the Euclidean distance $d(R_\alpha, \Omega^\complement)>0$.
We claim that $$\lim_{n\to \infty}\|\psi(z_n,w_n) - (z_n,w_n)\|= 0.$$
Indeed,
$$\|\psi(z_n,w_n) - (z_n,w_n)\|\leq \sum_{j=0}^\infty \|\psi_{j+1}(z_n,w_n)-\psi_j(z_n,w_n)\|=\sum_{j=0}^\infty\|\delta_j(z_n,w_n)\|\leq
 \sqrt 2 \sum_{j=0}^\infty (j+1)e^{-\Re(z_{n+j} )},$$
and the claim follows since ${\rm Re}\, z_{n+j}\geq {\rm Re}\, z_0+(n+j)\alpha$.
Hence, if $n$ is large enough, $\psi(z_n,w_n)\in \Omega$.
Since $\Omega$ is completely invariant under $L$ and $U$ is completely invariant under $F$, it follows that $\psi(U)\subset \Omega$.

What is left is to show that $\Omega\subset \psi(U)$. Let $(x_0, y_0) \in \Omega$, and write $(x_n, y_n) := L^n(x_0, y_0)$. By definition of $\Omega$ we have that $\beta := \mathrm{Re}(x_0 - y_0) > 0$. Note that $\mathrm{Re}(x_n - y_n)$ also equals $\beta$, and that $\mathrm{Re}(x_n), \mathrm{Re}(y_n) \rightarrow \infty$. Let $0< \alpha < \beta$. Recalling that $\eta_\alpha(x) \rightarrow 0$ as $x \rightarrow +\infty$ it follows that $(x_n, y_n) \in R_\alpha$ for all $n\in \mathbb N$ sufficiently large. In fact, there exists an $r>0$ and an $N \in \mathbb N$ such that $R_\alpha$ contains the closed ball  $\overline B((x_n, y_n),r)$ for all $n \ge N$.
We claim that $$\lim_{n\to \infty}\|\psi - \mathrm{Id}\|_{\overline B((x_n, y_n),r)} = 0.$$
Indeed, for all $n\in \N$,
 $$\|\psi- \mathrm{Id}\|_{\overline B((x_n, y_n),r)}\leq \sum_{j=0}^\infty \|\psi_{j+1}-\psi_j\|_{\overline B((x_n, y_n),r)}=\sum_{j=0}^\infty\|\delta_j\|_{\overline B((x_n, y_n),r)}\leq
 \sqrt 2 \sum_{j=0}^\infty (j+1)\|e^{-\Re( \pi_1\circ F^j)}\|_{\overline B((x_n, y_n),r)}.$$
Assume now that $n\geq N$. Since $\overline B((x_n, y_n),r)\in R_\alpha$ we have that  for all $(x,y)\in \overline B((x_n, y_n),r)$,
$$e^{-\Re \pi_1(F^j(x,y))}\leq e^{-\Re x-j\alpha}\leq e^{-\Re x_0 -(n+j)\alpha+r},$$ where the last inequality follows from the fact that for all $n\in \N$ and $(x,y)\in \overline B((x_n, y_n),r)$ we have $\Re x\geq\Re x_0 +n\alpha -r$. This proves the claim.

Let $n \ge N$ be such that $\|\psi - \mathrm{Id}\| \le \frac{r}{2}$ on $\overline B((x_n, y_n),r)$.
By Rouch\'e's Theorem in several complex variables it follows that $(x_n, y_n) \in \psi(\overline B((x_n, y_n),r)) \subset \psi(U)$. Since $\Omega$ is completely invariant under $L$ and $U$ is completely invariant under $F$, it follow that $(x_0, y_0) \in \psi(U)$.

\end{proof}

\section{Escaping wandering domain}\label{sectionescaping}

\begin{defn}
Let $F$ be a transcendental H\'enon map.
A Fatou component $\Omega$ is a  \emph{ wandering domain} if it is not preperiodic.
A wandering domain
\begin{enumerate}
\item is  \emph{escaping} if all orbits converge to the line at infinity,
\item is   \emph{oscillating} if there exists an unbounded orbit and an orbit with a bounded subsequence,
\item is  \emph{orbitally bounded } if every orbit is bounded.
\end{enumerate}
\end{defn}
{For polynomials in $\C$ it is known  that wandering domains cannot exist \cite{Sullivan}.  For transcendental functions there are examples of  escaping wandering domains. \cite{BerSur} uses for example the function $f(z)=z+\lambda \sin (2\pi z)+1$ for suitable $\lambda$.
There are also examples of oscillating wandering domains (\cite{EL87}, \cite{Bi}), and it is an open question whether orbitally bounded  wandering domains can exist.

It follows from the existence of the filtration that a polynomial H\'enon map does not admit any escaping or oscillating wandering domain. In the remainder of the paper we will give examples of both escaping and oscillating wandering domains for transcendental H\'enon maps. The existence of orbitally bounded wandering domains  is an open question for both polynomial and transcendental H\'enon maps.

We start with the escaping case, and we will be inspired by the construction of escaping wandering domains for transcendental functions.
Similar to \cite{BerSur} we will use  functions of the form
$$
\tilde{f}(z) = z + \sin(2\pi z) + \lambda
$$
with appropriate values of $\lambda$.  It will be convenient for us to take the constant $\lambda$ such that we obtain an escaping orbit, all consisting of critical points. Note that
\begin{equation*}
\tilde{f}^\prime(z) = 1 + 2\pi \cos(2\pi z).
\end{equation*}
The critical points of $\tilde f$ are therefore the points that satisfy
\begin{equation}\label{Critical}
\cos(2\pi z) = \frac{-1}{2\pi}.
\end{equation}

A computation gives that $\tilde{f}$ has  two distinct bi-infinite sequences of critical points

\begin{align*}
c_{n,+}&:=\alpha+n &\text{  with $\sin(2\pi c_{n,+}) = + \sqrt{1 - \frac{1}{4\pi^2}}$}\\
c_{n,-}&:=-\alpha+n &\text{ with $\sin(2\pi c_{n,-}) = - \sqrt{1 - \frac{1}{4\pi^2}}$}
\end{align*}

for the appropriate value of $\frac{1}{4}<\alpha<\frac{1}{2}$ given by solving (\ref{Critical}).

Note that the set of critical points is invariant under translation by $1$.
By taking
$$
\lambda = 1 - \sqrt{1 - \frac{1}{4\pi^2}},
$$
$\tilde{f}$ acts as a translation by $1$ on the sequence $c_{n,+}$.

For simplicity of notation we consider the map $f$ obtained by conjugating $\tilde{f}$ with a real translation by $\alpha+\frac{1}{4}$, so that the critical points $c_{n,+}$ for $\tilde{f} $ are mapped to critical points $z_n=n$ for $f$. Thus, $f$ will act on $\mathbb Z$ as a translation by $1$, and $f$ commutes with translation by $1$.

Consider the function $g(z) := f(z) - 1$, which commutes with $f$ and with the translation by $1$. For the function $g$ each point $z_n$ is a  super-attracting \emph{fixed point}.
For all $n\in \Z$ denote by $B_n$ the immediate basin  for $g$ of the point $z_n$. Clearly the basins $B_n$ are disjoint,  $B_{n+1}=B_n+1$, and $f(B_n)\subset B_{n+1}$. If one shows that each $B_n$ is also a Fatou component for $f$, then clearly each $B_n$ is an escaping wandering domain for $f$.


There are two classical ways in one dimensional complex dynamics to show that each $z_n$ belongs to a different Fatou component for such a function $f$. One is   by constructing curves in the Julia set that separate the points $z_n$ \cite[p. 183]{De}, and the other one is by showing that any two commuting maps which differ by a translation have the same Julia set \cite[Lemma 4.5]{Ba}. The first kind of argument is typically unavailable in higher dimensions. The proof we present is an {\sl ad hoc} version of the second argument.
We will define transcendental H\'enon maps $F,G:\Ctwo\ra\Ctwo$ which behave similarly  to $f,g$, then use $G$ to show that the  norm of the differentials   $dG^n$ and $dF^n$ at points on the boundaries of the attracting basins for $G$ grow exponentially in $n$,
and finally  use the latter information to  imply that those boundaries must be contained in the Julia set for $F$. Since the attracting basins are disjoint for $G$, the corresponding sets are disjoint for $F$, giving a sequence of wandering domains.

Let us define the H\'enon map
$$
F(z,w) = (f(z) + \delta(z-1) - \delta w, z),
$$
where $f$ is as constructed before and  $\delta>0$ is some constant to be suitably chosen later. Since $f$ commutes with translation by $1$, the map $F$ commutes with translation by $(1,1)$. Moreover, on the sequence of points $(n, n - 1)$ the map $F$ acts as a translation by $(1,1)$. We also define the map
$$
G(z,w) := F(z,w) - (1,1).
$$
Hence the points $P_n=(n, n-1)$ are all fixed points of $G$. Since the points $z_n$ are critical points of $f$, we can choose $\delta$ sufficiently small so that the fixed points $P_n=(n, n-1)$ are attracting for $G$. Denote by $A_{n}$ the attracting basin of the point $P_n$, which is biholomorphic to $\C^2$. It is easy to see that $A_{n+1}=A_n+(1,1)$ and that    $F(A_n)= A_{n+1}$. We claim that each $A_n$ is also a Fatou component for $F$, and thus each $(A_n)$ is an escaping wandering domain. We first need  a Lemma.
\begin{lemma}\label{lemma:smallderivatives}
Let $0<\lambda<1$, and let $(\xi_n\colon \D\to \C)$ be a sequence of holomorphic functions satisfying
$$
|\xi_n(z)| < \lambda^n
$$
for all $z \in \mathbb D$.  Let $0<\beta<1$. Then
$$
|\xi_n^{(k)}(0)| < k!\cdot\beta^k
$$
for
$$
k < \frac{n}{\log_\lambda(\beta)}.
$$
\end{lemma}
\begin{proof}
It follows from the Cauchy estimates since the assumption on $k$ is equivalent to
$$
\lambda^n < \beta^k.
$$
\end{proof}

\begin{thm}
Let $F$ be a transcendental H\'enon map that commutes with a translation $T = T_{(1,1)}$. Write
$$
G :=  T^{-1} \circ F,
$$
assume that $G$ has an attracting fixed point $Q$, and denote its basin of attraction by $A$. Then $A$ is also a Fatou component of $F$.
\end{thm}
\begin{proof}
Notice that $(T^k(Q))$ is an $F$-orbit  converging to the point $[1:1:0]$ on the line at infinity. Since  $F^n=T^n\circ G^n$ for all $n\in \N$, it follows that the $F$-orbits of all points in $A$ converge to $[1:1:0]$. Hence $A$ is contained in a Fatou component $\Omega$ of $F$. We show that $\Omega$ equals $A$ by proving that all boundary points of $A$ are contained in the Julia set  of $F$.

Without loss of generality we may assume that the point $Q$ is the origin.
Let $P \in \partial A$. Since $P\notin A$, its $G$-orbit avoids some definite ball centered at $(0,0)$, hence there exists $0<\mu<1$ for which
$$
\|G^n(P)\| \ge \sum_{k=0}^\infty \mu^k
$$
for all $n \in \mathbb N$.
For each $n$ we can choose an orthogonal projection $\pi_n$ onto a complex line through $(0,0)$ so that
$$
|\pi_n(G^n(P))| \ge \sum_{k=0}^\infty \mu^k.
$$

Let $\varphi: \mathbb C \rightarrow \mathbb C^2$ be an affine embedding for which $P \in \varphi(\mathbb{D})$ and $\varphi(0) \in A$. We will show that for any choice of $\varphi$ the derivatives of $G^n\circ \varphi$ grow exponentially fast for some point in the disk $D(0,\eta)$, where $\eta>\frac{1}{\mu}$ is chosen independently of $\varphi$. This will imply that some point in $\varphi(D(0,\eta))$ is contained in the Julia set of $F$. Since $\varphi$ is arbitrary,
one can choose $\varphi(D(0,\eta))$ to be contained in  arbitrarily small neighborhoods of $P$, giving a sequence of points $Z_n\ra P$ belonging to the Julia set of $F$. Since the latter is closed, the statement of the theorem follows.

We consider the sequence of maps $\psi_n: \mathbb C \rightarrow \mathbb C$ defined by
$$
\psi_n := \pi_n \circ G^n \circ \varphi.
$$

Let $r>0$ be  such that  $\varphi(D(0,r))\subset \subset A_0$. Then there exist $C>0$ and $\lambda<1$ so that
$$
\|\psi_n\|_{D(0,r)} < C \cdot \lambda^n.
$$
By instead defining
$$
\psi_n = \pi_n \circ G^{n+N} \circ \varphi.
$$
for some large integer $N \in \mathbb N$ we may assume that
$$
\|\psi_n\|_{D(0,r)} < \lambda^n.
$$
By defining the maps $\xi_n(z) := \psi_n(r\cdot z)$ we obtain a sequence of maps $(\xi_n\colon \D\to \C)$ that satisfy the conditions of Lemma \ref{lemma:smallderivatives}, which we apply with $\beta:=\mu\cdot r$. Hence
$$
|\xi_n^{(k)}(0)| < k!\mu^kr^k \text{\ \ \ \ for $k < \frac{n}{\log_\lambda(\mu r)}$},
$$
which implies that
\begin{equation}\label{Contracting estimates}
|\psi_n^{(k)}(0)| < k! \mu^k \text{\ \ \ \ for $k < \frac{n}{\log_\lambda(\mu r)}$}.
\end{equation}
Writing $\zeta = \varphi^{-1}(p)\in \D$, we also have that
\begin{equation}\label{psiestimates}
|\psi_n(\zeta)| \ge \sum_{k=0}^\infty\mu^k
\end{equation}
for all $n$. Writing
$$
\psi_n(z) = \sum_{k=0}^\infty a_k z^k,
$$
by (\ref{psiestimates}) there is at least one $k\in \N$ for which
$$
|a_k|\ge \mu^k,
$$
and by (\ref{Contracting estimates}),
  $k \geq \frac{n}{\log_\lambda(\mu r)}$. Putting things together we get
$$
|\psi_n^{(k)}(0)| \ge k!\mu^k \text{\ \ \  for some $k \geq \frac{n}{\log_\lambda(\mu r)}$}.
$$
Let $\eta>\frac{1}{\mu}$. By Cauchy estimates  this implies the existence of a $z_* \in  D(0,\eta)$ for which
$$
|\psi_n^\prime(z_*)| > \frac{\eta^{k-1} \mu^k k!}{(k-1)!}\geq C \Lambda^n,
$$
where $\Lambda=(\eta \mu)^\frac{1}{\log_\lambda(\mu r)}>1 $ and $C>0$ is constant. Denote $P_*:=\varphi(z_*)$. Since $\|\pi_n\|=1$  we get, up to changing $C$,
$$
 \|d_{P_*}(G^n)\|\geq  C \Lambda^n.
$$
It follows from $F^n = T^n \circ G^n$  that
\begin{equation}\label{Ffast}
  \|d_{P_*}(F^n)\|  \geq C \Lambda^n.
\end{equation}

We now show that a point in $\varphi( D(0,\eta))$ is contained in the Julia set of $F$.
Assume for the purpose of a contradiction that $\varphi( D(0,\eta))$ is contained in the Fatou component $\Omega$.
It follows that $F^n\rightarrow [1:1:0]$ uniformly on a neighborhood $U$ of $P_*$.

Consider the affine chart of $\P^2$ defined as $[z:w:t]\mapsto(w/z,t/z)$ defined on $\{z\neq 0\}$. On $\{z\neq 0\}\cap \{t\neq 0\}$ such chart has the expression $H(z,w)=(w/z,1/z)$.
The sequence $(H\circ F^n\colon U\to \C)$ is  defined for $n$ large enough  and converges uniformly to the point $(1,0)\in \C^2$.
Hence, denoting   $P_*:=(z_0,w_0)$,
\begin{equation} \label{chaincontradiction}
d_{(z_n,w_n)}H\circ d_{(z_0,w_0)} (F^n)      =d_{(z_0,w_0)} (H\circ F^n)\to 0.
\end{equation}

%

We have that
$$
(d_{(z_n,w_n)}H)^{-1}=\left(\begin{array}{cc}
0& -z_n^2\\
 z_n& -w_nz_n
  \end{array}\right).
$$
Since $(z_n,w_n)\to [1:1:0]$, arguing as in (\ref{zwartepiet}) we obtain that for every $\epsilon>0$ there exists $C'>0$ such that
$$\frac{1}{\| (d_{(z_n,w_n)}H)^{-1}\|}\geq \frac{1}{C'}\frac{1}{(1+\epsilon)^n},$$
and thus
$$\|d_{(z_n,w_n)}H\circ d_{(z_0,w_0)} (F^n) \|\geq \frac{\|d_{(z_0,w_0)} (F^n)\|}{\|(d_{(z_n,w_n)}H)^{-1}\|}\geq  \frac{C}{C'}\frac{\Lambda^n}{(1+\epsilon)^n}.$$
Hence if $\epsilon$ is small enough, this contradicts (\ref{chaincontradiction}).

\end{proof}

\section{Oscillating wandering domain}\label{sectionoscillating}
We show in this section the existence of a transcendental H\'enon map with an oscillating wandering domain biholomorphic to $\C^2$.

Notice that, up to a linear change of variables, any   transcendental H\'enon map can be written in the alternative form
$F(z,w)=(f(z)+aw,az)$, where $f$ is a transcendental function and $a\neq 0$.
We will consider maps of the form  $$F(z,w):=(f(z)+aw,az), \quad f(z):=bz+O(|z|^2), $$ where $a,b\in \R$, $a<1$, are such that the origin is a saddle point. For example,
since the eigenvalues of $d_0 F$ are
 $$
\lambda= \frac{b\pm \sqrt{b^2+4a^2}}{2},
$$
we can pick $a=1/2$, and $b=1$ to obtain $\lambda_s=\frac{1- \sqrt{2}}{2}$,  $\lambda_u=\frac{1+ \sqrt{2}}{2}$.

We will construct $F$ as the uniform  limit on compact subsets of a sequence of automorphisms $F_k(z,w):=(f_k(z)+aw,az)$ with the same   value for $a$ and $b$. We note that for transcendental H\'enon maps $F(z,w)=(f(z)+aw, az), G(z,w)=(g(z)+aw, az)$ and any subset $A\subset \C$ we have
$$
\|f-g\|_A=\|F-G\|_{A\times \C}.
$$

\begin{prop}\label{prop:claim1} There exists a sequence of entire  maps
$$
F_k(z,w)=(f_k(z)+aw,az), \quad f_k(z)=bz+O(|z|^2), \quad k=0,1,2,\dots
$$
a sequence of points $P_n=(z_n,w_n)$, where $n=0,1,2,\dots $,  sequences $R_k\ra\infty$, $0<\epsilon_k\leq \frac{1}{2^k}$ and $\beta_n\to 0$, a decreasing  sequence
$\theta_k\to 0$,  and strictly increasing sequences of integers $\{n_k\}, \{n'_k\}$ with $n_k<n'_k< n_{k+1}$, such that the following five properties are satisfied:
\begin{itemize}
\item[(i)] $\|F_{k}-F_{k-1}\|_{D(0,R_{k-1})\times\C}\leq \epsilon_{k}$ for all $k\geq 1$;
\item[(ii)] $F_{k}(P_n)=P_{n+1}$ for all $0 \leq n < n_{k}$;
\item[(iii)]\label{condizioneintorni}
$F_k(B(P_n,\beta_n))\subset \subset B(P_{n+1},\beta_{n+1})$ for all $0 \leq n<n_k$, where $\beta_n\leq \frac{\theta_{\rho}}{2}$ for all $0\leq n\leq n_k$;
 \item[(iv)]  \label{DisjointPoints}  $|z_n|<R_{k} -\theta_{\rho}$ for all $0\leq n< n_k$, and  $|z_{n_{k}}|> R_{k} +5 \theta_k$;
\item[(v)] $P_{j}\in B(0,\frac{1}{k})$  for some $j$ with $n_{k-1}\leq j\leq n_{k}$, for all $k\geq 1$.
\end{itemize}
Here $\rho:=\rho(n)$ denotes the unique integer for which $n'_{\rho-1}\leq n< n'_\rho$, using $n'_{-1}=0$.
\end{prop}

Before proving this proposition, let us show that it implies the existence of a wandering Fatou component.
By (i), the maps $F_k$ converge uniformly on compact subsets to a holomorphic automorphism $F(z,w)=(f(z)+aw,az)$. By (ii), (iv) and (v) it follows that  $(P_n)$ is an oscillating orbit for $F$, that is, it is unbounded and it admits a subsequence converging to the origin.
 Since by (iii) for all $j\in \N$ the iterate images of $B(P_j,\beta_j)$ have uniformly bounded Euclidean diameter (in fact, the diameter goes to zero), each ball $B(P_j,\beta_j)$ is contained in the Fatou set of $F$.

\begin{lemma}
If $i\neq j$, then $P_i$ and $P_j$ are in different Fatou components of $F$, and hence  $F$ has an    oscillating wandering domain.
\end{lemma}
\begin{proof}

For all $j\in \N$, denote $\Omega_j$ the Fatou component containing $B(P_j,\beta_j)$. Since $\beta_n\to 0$ as $n\to\infty$, by (iii) and by identity principle it follows that all limit functions on each $\Omega_j$ are constants.

Assume $j > i$ and let $k = j-i$.
Let $\mathcal{N}(0)$ be a neighborhood of the origin that contains no periodic points of order less than or equal to $k$. Since the orbit $(P_n)$ enters and leaves a compact subset of $\mathcal{N}(0)$ infinitely often, there must be a subsequence $(n_j)$ for which $P_{n_j}$ converges to a point $z \in \mathcal{N}(0) \setminus \{0\}$. But then $P_{n_j + k}$ converges to $F^k(z) \neq z$, which implies that $P_i$ and $P_j$ cannot be contained in the same Fatou component.
\end{proof}

We will prove Proposition \ref{prop:claim1} by induction on $k$. We  start the induction  by letting $R_0:=1$, $\theta_0:=1$, $\beta_0:=\frac{1}{2}$ and  $P_0:=(z_0,w_0)$ with $|z_0| > 6$. We set $F_0(z,w):= (bz+aw,az)$ and $n_0:=0$.

Let us now suppose that (i)-(v) hold for certain $k$, and let us proceed to define $F_{k+1}$ and the  points $(P_n)_{n_{k}< n\leq n_{k+1}}$. This is done in two steps. The first step relies upon the classical Lambda Lemma \ref{lambdalemma}:

\begin{lemma}[Lambda Lemma, see e.g. \cite{PdM82}]\label{lambdalemma}
Let $G$ be a holomorphic automorphism of $\mathbb C^2$ with a saddle fixed point at the origin. Denote by $W^s(0)$ and $W^u(0)$  the stable and unstable manifolds respectively.
Let $p \in W^s(0)\setminus\{0\}$ and $q \in W^u(0)\setminus\{0\}$, and let $D(p)$ and $D(q)$ be holomorphic disks through $p$ and $q$, respectively transverse to $W^s(0)$ and $W^u(0)$. Let  $\epsilon >0$. Then there exists $N\in \N$  and  $N_1>2N+1$, and  a point   $x \in D(p)$ with $G^{N_1}(x) \in D(q)$  such that  $\|G^n(x) - G^n(p)\| < \epsilon$ and $\|G^{N_1-n}(x) - G^{-n}(q)\|< \epsilon$ for $0\leq n \le N$, and $\|G^n(x)\| < \epsilon$ when $N < n < N_1 - N$.
\end{lemma}

Using the Lambda Lemma we find a finite $F_k$-orbit $(Q_j)_{0\leq j\leq M}$, the \emph{new oscillation}, which goes inward along the stable manifold of $F_k$, reaching the ball $B(0,\frac{1}{k+1})$, and then outwards along the unstable manifold of $F_k$.

The second step of the proof relies upon another classical result:
\begin{lemma}[Runge approximation]\label{rungeappr}
Let $K\subset \C$ be a polynomially convex compact subset (recall that $K$ is polynomially convex  if and only if $\C\setminus K$ is connected). Let $h\in \mathcal{O}(K)$, and let $\{p_i\}_{0\leq i\leq q}$ be a set of points in $K$. Then for all $\epsilon >0$ there exists an entire holomorphic function $f\in \mathcal{O}(\C)$ such that $$\|f-h\|_K\leq \epsilon$$ and such that
$$f(p_i)=h(p_i),\, f'(p_i)=h'(p_i) \quad \forall \,0\leq i\leq q.$$
\end{lemma}

Using Runge approximation we find a map $F_{k+1}$ connecting the previously constructed finite orbit  $(P_n)_{0\leq n\leq n_{k}}$ with the new oscillation $(Q_j)_{0\leq j\leq M}$ via a finite orbit along which $F_{k+1}$ is sufficiently contracting. The contraction neutralizes possible expansion along the new oscillation $(Q_j)_{0\leq j\leq M}$, and we refer to this connecting orbit as the \emph{contracting detour}.

\begin{figure}[t]
\centering
\includegraphics[width=2.5in]{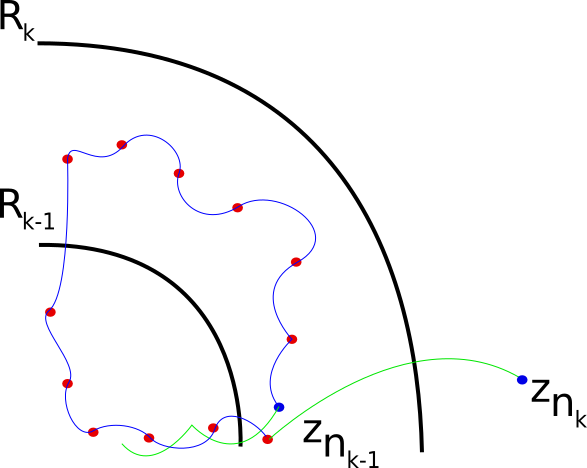}
\caption{The first coordinates of the orbit $(P_n)$.}
\label{figure:oscillation}
\end{figure}

\subsection{Finding the new oscillation}

\begin{lemma}\label{lemma:claim2} There exist a finite $F_k$-orbit $(Q_j):=(z_j',w_j')_{0\leq j\leq M}$ intersecting the ball $B(0,\frac{1}{k+1})$ and a small enough  $\theta_{k+1}>0$   such that
the three disks
$$
\overline D(z_{n_k},\theta_{k}),  \quad  \overline D(w'_0/a,\theta_{k+1}), \quad \overline D(z'_{M},\theta_{k+1})
$$
are pairwise disjoint, and disjoint from the polynomially convex set
$$
K:=\overline D(0, R_k)\cup  \bigcup_{0\leq i< M}\overline D(z'_i,\theta_{k+1}).
$$
\end{lemma}
\begin{proof}
Let $(\varphi^s,\psi^s)\colon \C\to W^s_{F_k}(0,0)$ be  linearizing coordinates for the stable manifold of the map $F_k$.
Let $\zeta\in \C$ be a point with minimal absolute value such that
\begin{equation}\label{conditionstable}
|\psi^s(\zeta)|= a(|z_{n_{k}}| - 4\theta_{k}).$$  It is easy to see that $$|\psi^s(\eta)|< |\psi^s(\zeta)|, \quad \forall \eta\colon |\eta|< |\zeta|.
\end{equation}
Setting $(u_0, v_0):=(\varphi^s(\zeta),\psi^s(\zeta))$, by definition we have
$$\frac{|v_0|}{a}=|z_{n_{k}}| - 4\theta_{k}.$$ Moreover, the forward orbit  $(u_{i}, v_{i})$ satisfies
$$
|v_i| =|\psi^s(\lambda_s^i\cdot \zeta)|< |v_0|, \quad \forall i\geq 1,
$$
and since $|u_i| = \frac{|v_{i+1}|}{a}$ it follows that
$$
|u_i| < |z_{n_{k}}| - 4\theta_{k}, \quad \forall i\geq 0.
$$

Similarly, we find a point $(s_0, t_0)$ in the unstable manifold $W^u_{F_{k}}(0,0)$ whose backwards orbit $(s_{-i}, t_{-i})$ satisfies $|s_{-i}| < |s_0|$ and for which $$|s_0| = |z_{n_{k}}| - 2\theta_{k}.$$
By taking an arbitrarily small perturbation of $(s_0, t_0)$, we can make sure that the discrete sequence $(s_i)$ avoids the value $\frac{v_0}{a}$.

Consider arbitrarily small disks $D_1$ centered at $(u_0, v_0)$ and $D_2$ centered at $(s_0, t_0)$ transverse to the stable respectively the unstable manifold.
It follows from the Lambda Lemma  \ref{lambdalemma} that there exists a finite $F_k$-orbit  $(Q_j):=(z_j',w_j')_{0\leq j\leq M}$ which intersects the ball $B(0,\frac{1}{k+1})$ with $Q_0\in D_1$ and $Q_{M}\in D_2$.
If all the perturbations are chosen small enough, then the sequence  $(z'_j)_{0\leq j\leq M}$ avoids the disk $\overline D(z_{n_k},\theta_{k})$ and the point  $\frac{w'_0}{a}$.
Setting $\theta_{k+1}>0$  small enough completes the proof.
\end{proof}

\subsection{Connecting the orbits via the contracting detour}
By continuity of $F_k$ there exist constants $0<\tilde \beta_j\leq \frac{\theta_{k+1}}{2}$  for all $0\leq j\leq M$ such that  $$F_k(B(Q_j,\tilde \beta_j))\subset \subset F_k(B(Q_{j+1},\tilde \beta_{j+1})),\quad 0\leq j<M.$$
To control the contraction we will need the following Lemma, a direct consequence of the formula of transcendental H\'enon maps.

\begin{lemma}\label{lip}
Let $f\in \mathcal{O}(\C)$ and let  $F(z,w):=(f(z)+aw,az)$.
Choose $a',a''$ such that $$0<a''<a<a'<1,$$ and let $D_1\subset \subset D_2\subset \C$ two open disks.
Then there exists $\alpha>0$ such that if we have
$\|f-s\|_{\overline D_2}\leq \alpha,$
for some constant $s\in \C$, then
$$ a'' \|(z,w)-(z',w')\|\leq \|F(z,w)-F(z',w')\|\leq a' \|(z,w)-(z',w')\|, \quad \forall z,z'\in \overline D_1.$$
\end{lemma}


Fix $0<a''<a<a'<c<1$ and let $N>0$ be such that
\begin{equation}\label{fattorecontrazione}
c^N\cdot \beta_{n_k}< \tilde \beta_{0}.
\end{equation}

We now construct the contracting detour, starting at $P_{n_k}$ and ending at $Q_0$, obtaining a contraction by at least the factor $c^N$.
Choose a family of points $\{z''_j\}_{0\leq j\leq N}\subset \C$ such that
\begin{enumerate}
\item $z''_0=z_{n_k}, z''_{N}=\frac{w_0'}{a},$
\item $|z''_j|>|z_{n_k}|+2$ for all $1\leq j\leq N-1$,
\item $|z''_j-z''_i|>2$,  for all $0\leq i\neq j\leq N-1$.
\end{enumerate}

\medskip

Let $R_{k+1}>0$ be such that
\begin{equation}\label{almostdone}
R_{k+1}>|z''_j| +\theta_{k},\ \forall 0\leq j\leq N ,\quad \mbox {and}\quad R_{k+1}>|z'_j| +\theta_{k+1}, \ \forall 0\leq j\leq M.
\end{equation}

 Define $w''_0:=w_{n_k}$ and $w''_j:=az''_{j-1}$ for all $1\leq j\leq N$, and consider the points $\{T_j\}_{1\leq j\leq N}\subset \C^2$ defined as $$T_j:=(z''_j,w''_j).$$

By the choice of the points $(z''_j)$ and by Lemma \ref{lemma:claim2} it follows that the disks
$$
(\overline D(z''_j,\theta_{k}))_{0\leq j< N}, \quad  \overline D(w'_0/a,\theta_{k+1}),\quad \overline D(z'_{M},\theta_{k+1})
$$
are pairwise disjoint, and disjoint from the polynomially convex set $K$. Let $H$ denote the union of such disks.

 We  define a holomorphic function on the polynomially convex set $K\cup H$ in the following way:
 \begin{enumerate}
\item  $h$ coincides with $f_k$ on $K$,
\item $h|_{\overline D(z''_j,\theta_k)}$ is constantly equal to  $z''_{j+1}-aw''_j$ for all $0\leq j< N$,
\item $h|_{\overline D(w'_0/a,\theta_{k+1})}$ is constantly equal to $ z_0'-aw''_N,$
\item $h|_{\overline D(z'_{M},\theta_{k+1})}$ is constantly equal to some value $A>R_{k+1}+5\theta_{k+1}.$
 \end{enumerate}

 By the Runge approximation Lemma \ref{rungeappr} there exists a function $f_{k+1}\in \mathcal{O}(\C)$ such that
\begin{enumerate}
\item   $f_{k+1}(0)=h(0)=0, f_{k+1}'(0)=h'(0)=b,$
\item    $f_{k+1}(z_j)= h(z_j)$ for all $0\leq j< n_k$,
\item  $f_{k+1}(z'_j)= h(z'_j)$ for all $0\leq j\leq  M$,
\item $f_{k+1}(z''_j)= h(z''_j)$ for all $0\leq j\leq  N$,
\item $\|f_{k+1}-h\|_{K\cup H}<\epsilon_{k+1}.$
\end{enumerate}
where  $\epsilon_{k+1}\leq \frac{1}{2^{k+1}}$, and $\epsilon_{k+1}$ is also  smaller than the minimum of the constants  $\alpha$ given by  Lemma \ref{lip} for the following pairs of disks:
\begin{enumerate}
\item $D(z''_j,\frac{\theta_k}{2})\subset \subset D(z''_j,\theta_k) $ for all $0\leq j< N$,
\item $\overline D(w'_0/a,\frac{\theta_{k+1}}{2})\subset \subset \overline D(w'_0/a,\theta_{k+1})$,
\item  $\overline D(z'_{M},\frac{\theta_{k+1}}{2})\subset \subset \overline D(z'_{M},\theta_{k+1})$.
\end{enumerate}

Define $F_{k+1}(z,w):=(f_{k+1}(z)+aw,az).$
It is easy to see that the  sequences of points  $$(P_n)_{0\leq n\leq n_{k}}, (T_j)_{1\leq j\leq N},(Q_j)_{0\leq j\leq M}$$ form an $F_{k+1}$-orbit,
that is, the contracting detour $(T_j)_{1\leq j\leq N}$ connects the old and new pieces of orbit.

Set $$n'_k:=n_k+N+1,\quad n_{k+1}:=n'_k+M+1,$$ so that  $P_{n'_k}:=Q_0$, and  $P_{n_{k+1}}:=F_{k+1}(Q_M)$.
Define $$\beta_{n_k+\ell}:=\beta_{n_k}\cdot c^\ell, \quad \forall \,0\leq \ell\leq N, $$
$$\beta_{n'_k+\ell}:=\tilde \beta_\ell, \quad \forall\, 0\leq \ell\leq M.$$

It is easy to see that, up to taking a smaller $\epsilon_{k+1}$, the map $F_{k+1}$ satisfies property (iii).
Property (iv) follows from (\ref{almostdone}) and from $A>R_{k+1}+5\theta_{k+1}.$
Since by construction the new piece of orbit  $(Q_j)$ intersects the ball $B(0,\frac{1}{k+1})$, property (v) is satisfied.
Thus Proposition \ref{prop:claim1} is proved, completing the proof of the existence of an oscillating wandering domain.

\subsection{Complex structure}
We will now prove that, by making the contracting detours sufficiently long, the oscillating wandering domains can be guaranteed to be biholomorphically equivalent to $\mathbb C^2$. We denote the Fatou component containing $P_0$ by $\Omega$.

We have constructed an orbit $(P_n)$ and a sequence of radii $(\beta_n)$ such that
$$
F (B(P_n, \beta_n)) \subset\subset B(P_{n+1}, \beta_{n+1}).
$$
Define the  \emph{calibrated basin}
$$
\Omega_{(P_n), (\beta_n)} := \bigcup _{n\in \N}F^{-n} (B(P_n, \beta_n)),
$$
and notice that it is contained in $\Omega$.

\begin{lemma}
We can guarantee that the calibrated basin $\Omega_{(P_n), (\beta_n)} $ is biholomorphic to  $\mathbb C^2$.
\end{lemma}
\begin{proof}
For all $n\geq 0$, let $H_n\in {\rm Aut}(\C^2)$ be defined by $H_n(z):=P_n+  \beta_n \cdot z$.
For all $m\geq n\geq 0$, define $$\tilde F_{m,n}:=    H_m^{-1}\circ F^{m-n}\circ H_n.$$
Then for all $n\geq 0$, we have that $\tilde F_{n+1,n}(\B^2)\subset \B^2$ and $\tilde F_{n+1,n}(0)=0.$
It is easy to see that the calibrated basin  $\Omega_{(P_n), (\beta_n)} $ is biholomorphic to the set
$\tilde \Omega:=\bigcup_{n\in \N}\tilde F_{n,0}^{-1}(\B^2).$

If $n\in \N$  belongs to a contracting detour, then
\begin{equation}\label{eq:contraction}
\frac{a''}{c}\|x\|\leq\|\tilde F_{n+1,n}(x)\|\leq \frac{a'}{c}\|x\|,\quad \forall x\in \B^2,
\end{equation}
where we can assume that $a'^2<a''$. If this was the case for every $n$ then it would follow immediately that the maps
$$
\Phi_n := (d_0\tilde{F}_{0,n})^{-1} \circ \tilde{F}_{0,n}
$$
converge to a biholomorphism from the calibrated basin to $\mathbb C^2$, see for example \cite{Wold}.

While we cannot guarantee that equation \eqref{eq:contraction} with $a'^2<a''$ holds outside of the contracting detours, by making the contracting detours sufficiently long we can still ensure that $\Phi_n$ converges uniformly on compact subsets to a biholomorphism $\Phi\colon \tilde \Omega\to \C^2$, compare for example Theorem 1.3 of \cite{PW}.

%
%
\end{proof}

We will now show, again by using the plurisubharmonic method, that the wandering Fatou component $\Omega$  can be forced to be equal to the calibrated basin.
Consider the sequence of plurisubharmonic functions $g_n$ on $\mathbb C^2$ given by
$$
g_n(z,w): = \frac{\log\|F^n(z,w) - P_n\|}{n}.
$$

\begin{lemma}
We can guarantee that the functions $g_n$ converge to $\log a$ on $\Omega_{(P_n), (\beta_n)}\setminus\{P_0\}$.
\end{lemma}
\begin{proof}
Take sufficiently many contractions that are sufficiently close to multiplication by $a$.
\end{proof}

Recall from induction hypothesis (v) in Proposition \ref{prop:claim1} that for every $k$ there exists an integer $j = j_k \in (n_{k-1}, n_k)$ with $\|P_j\| < \frac{1}{k}$. In particular the subsequence $(P_{j_k})$ is bounded. Since every limit function of $(F^n)$ on the Fatou component $\Omega$ is constant, it follows that for all $(z,w) \in \Omega$ we have
$$
\|F^{j_k}(z,w)-P_{j_k}\|\to 0
$$

As a consequence, we have that $(g_{j_k})$ is locally uniformly bounded from above and that for all $(z,w)\in \Omega$,
$$\limsup_{k\to\infty}g_{j_k}(z,w)\leq 0.$$

\begin{lemma}
We can guarantee that $g_{j_k} \rightarrow 0$ on $\Omega \setminus \Omega_{(P_n), (\beta_n)}$, uniformly on compact subsets.
\end{lemma}
\begin{proof}
Recall that in the construction of the wandering Fatou component the radius $\theta_{j_k}$ is determined before the length of the contracting detour. By making the contracting detour sufficiently long, we can guarantee that $j_k$ is as large as we want. For points $(z_0,w_0) \in \Omega \setminus \Omega_{(P_n), (\theta_n)}$ we have that
$$
\|(z_{j_k},w_{j_k}) - P_{j_k}\| \ge \theta_{j_k}.
$$
Thus, by making the contracting detours sufficiently long, we can guarantee that $g_{j_k}> -\epsilon_k$ on $\Omega \setminus \Omega_{(P_n), (\theta_n)}$ for any $\epsilon_k \searrow 0$. The conclusion follows.
\end{proof}

\begin{prop}
With the previous choices, the wandering Fatou component $\Omega$ equals the calibrated basin $\Omega_{(P_n), (\theta_n)}$, and is thus biholomorphically equivalent to $\mathbb C^2$.
\end{prop}
\begin{proof}
Suppose by contradiction that $\Omega\neq \Omega_{(P_n), (\theta_n)}$.
Let $h$ be the upper-semicontinuous regularization of $\lim_{k\to\infty} g_{n_k}.$ Clearly $h\equiv \log a$ on $\Omega_{(P_n), (\theta_n)}$ and  $h\equiv 0$ on $\Omega\setminus \Omega_{(P_n), (\theta_n)}$. Then by [Klimek, Prop 2.9.17] the function $h$ is plurisubharmonic, and  the submean value property  at any  $\zeta\in \partial  \Omega_{(P_n), (\theta_n)}$ gives a contradiction.
\end{proof}

\end{document}